\documentclass[a4paper, 12pt]{amsart}
\usepackage{hyperref, amssymb, tikz, amsmath, amsthm}
\usepackage{tikz-cd}

\usetikzlibrary{arrows,automata,positioning}

\newtheorem{thm}{Theorem}[section]
\newtheorem{prp}[thm]{Proposition}
\newtheorem{lem}[thm]{Lemma}
\newtheorem{cor}[thm]{Corollary}
\theoremstyle{definition}

\theoremstyle{remark}
\newtheorem{rmk}[thm]{Remark}

\numberwithin{equation}{section}

\newcommand{\Pp}{\mathcal{P}}

\newcommand{\CC}{\mathbb{C}}
\newcommand{\NN}{\mathbb{N}}

\newcommand{\TT}{\mathbb{T}}
\newcommand{\ZZ}{\mathbb{Z}}

\newcommand{\id}{\operatorname{id}}
\newcommand{\lsp}{\operatorname{span}}

\newcommand{\im}{\operatorname{Im}}
\newcommand{\coker}{\operatorname{coker}}

\newcommand{\tor}{\operatorname{tor}}
\newcommand{\rank}{\operatorname{rank}}

\newcommand{\lcm}{\operatorname{lcm}}

\title[The classification of some generalised Bunce--Deddens algebras]{The classification of some generalised Bunce--Deddens algebras}
\author{James Rout}
\address{School of Mathematics and Applied Statistics\\
University of Wollongong\\
Wollongong NSW 2522\\ Australia}
\email{jdr749@uowmail.edu.au}

\date{\today}
\thanks{This research is supported by an Australian Government Research Training Program (RTP) Scholarship.}

\begin{document}

\begin{abstract}
We use $K$-theory to prove an isomorphism theorem for a large class of generalised Bunce--Deddens algebras constructed by Kribs and Solel from a directed graph $E$ and a sequence $\omega$ of positive integers.  In particular, we compute the torsion-free component of the $K_0$-group for a class of generalised Bunce--Deddens algebras to show that supernatural numbers are a complete invariant for this class.
\end{abstract}

\maketitle

\section{Introduction}

In \cite{KribsSolel:JAMS07} Kribs and Solel introduced a family of direct limit $C^*$-algebras constructed from directed graphs $E$ and sequences $\omega = (n_k)_{k=1}^\infty$ of natural numbers such that $n_k | n_{k+1}$ for all $k \in \NN$. They called these $C^*$-algebras generalised Bunce--Deddens algebras. The graph $E$ consisting of a single vertex connected by a single loop-edge generates the classical Bunce--Deddens algebras.

Supernatural numbers have been used to classify UHF algebras (\cite[Theorem~1.12]{Glimm:TAMS60}) and the classical Bunce--Deddens algebras (\cite[Theorem~3.7]{BunceDeddens73} and \cite[Theorem~4]{BunceDeddens75}). Kribs showed in \cite[Theorem~5.1]{Kribs02} that the generalised Bunce--Deddens algebras corresponding to the graph $B_N$ consisting of a single vertex with $N$ loops, are classified by their associated supernatural numbers in the sense that $C^*(B_N,\omega) \cong C^*(B_N,\omega')$ if and only if $[\omega] = [\omega']$. The special case $N = 1$ is Bunce and Deddens' theorem. Kribs and Solel later showed in \cite[Theorem~7.5]{KribsSolel:JAMS07} that the generalised Bunce--Deddens algebras corresponding to the simple cycle with $j$ edges, are classified by their associated supernatural numbers; again the special case $j = 1$ is the original result of Bunce and Deddens. Kribs and Solel asked in \cite[Remark~7.7]{KribsSolel:JAMS07} for what class of graphs $E$ a similar classification theorem could be obtained. Here we prove that such a theorem can be obtained for the class of generalised Bunce--Deddens algebras corresponding to a given strongly connected finite directed graph $E$ such that $1$ is an eigenvalue of the vertex matrix, and the only roots of unity that are eigenvalues are the $\Pp_E$-th roots of unity, where $\Pp_E$ is the period of the graph $E$.

In \cite[Proposition~3.11]{RobertsonRoutSims} it was shown that if $[\omega]=[\omega']$ then $C^*(E,\omega) \cong C^*(E,\omega')$ for row-finite directed graphs $E$ with no sinks or sources. The main result of this article (Theorem~\ref{thm:classification}) shows that if $C^*(E,\omega) \cong C^*(E,\omega')$ then $[\omega] = [\omega']$  for strongly connected finite directed graphs $E$ such that $1$ is an eigenvalue of $A_E^t$ and such that the only roots of unity that are eigenvalues of $A_E^t$ are the $\Pp_E$-th roots of unity. We prove this by studying the torsion-free component of $K_0(C^*(E,\omega))$; we assume that $1$ is an eigenvalue of $A_E^t$ to ensure that this is nontrivial. The Perron--Frobenius theorem (see \cite[Theorem~8.2.1]{Gantmacher:MatrixTheory}) says that if $1$ is an eigenvalue of $A_E^t$, then the $\Pp_E$-th roots of unity are also eigenvalues of $A_E^t$. The hypothesis that these are the only roots of unity that are eigenvalues of $A_E^t$ is nontrivial. The \emph{nonnegative inverse eigenvalue problem} asks which sets of $n$ complex numbers $\lambda_1, \dots, \lambda_n$ occur as the eigenvalues of some $n \times n$ nonnegative matrix. Deep results of \cite{KimOrmesRoush} regarding this problem show that it is possible for any collection of roots of unity to appear as eigenvalues of a nonnegative matrix.

If $1$ is not an eigenvalue of $A_E^t$, then $K_0({C^*(E,\omega)})$ is purely torsion and another argument (perhaps along the lines of \cite[Theorem~5.1]{Kribs02}) will be needed. We have not addressed that case in this article.

We begin in Section~\ref{PerronFrobenius} with some calculations for the sums of powers of matrices and about cokernels. We show that the matrix $\sum_{i=0}^{n_k/l-1} (A_E^{il})^t$, where $l := \lim_{j \to \infty} \gcd(\Pp_E,n_j)$ and $\gcd(\Pp_E,n_k) = l$, is invertible if the only eigenvalues of $A_E^t$ are the $\Pp_E$-th roots of unity (Lemma~\ref{cmkdet}). We recall the equivalence relation $\sim_l$ on $E^0$ established in \cite[Lemma~4.2]{RobertsonRoutSims} to show that $\coker(1-A_E^l)^t \cong \bigoplus_{i=1}^l \coker(1-A_E^t)$ (Corollary~\ref{C*(E(l))1}).

In Section~\ref{sec:K1} we compute $K_1(C^*(E,\omega))$ for strongly connected finite directed graphs $E$ such that the only roots of unity that are eigenvalues of $A_E^t$ are the $\Pp_E$-th roots of unity. We show that the torsion-free component is isomorphic to $l$ copies of $K_1(C^*(E))$ (Theorem~\ref{K_1 limit}). We do this by showing that $\ker(1-A_E(n))^t \cong \ker(1-A_E^n)^t$ for $n \geq 1$ (Lemma~\ref{lem:keriso}), and by showing that $K_1(C^*(E(n_k)) \to K_1(C^*(E(n_{k+1}))$ induces the identity map on $\ker(1-A_E^l)^t \cong \bigoplus_{i=1}^l \ker(1-A_E^t)$ for all $k$ such that $\gcd(\Pp_E,n_k) = l$.

In Section~\ref{sec:K0} we compute the torsion-free component of $K_0(C^*(E,\omega))$  for strongly connected finite directed graphs $E$ such that $1$ is an eigenvalue of $A_E^t$ and such that the only roots of unity that are eigenvalues of $A_E^t$ are the $\Pp_E$-th roots of unity. We show that this group is isomorphic to $l$ copies of the torsion-free component of $K_0(C^*(E))$ adjoined the supernatural number $[\omega]$ associated to $\omega$ (Theorem~\ref{supchar}). We do this by showing that $\coker(1-A_E(n)^t) \cong \coker(1-A_E^n)^t$ for $n \geq 1$ (Lemma~\ref{cokeriso}), and by showing that the map $K_0(C^*(E(n_k)) \to K_0(C^*(E(n_{k+1}))$ induces the multiplication by $n_{k+1} / n_k$ map on $\coker(1-A_E^l)^t \cong \bigoplus_{i=1}^l \coker(1-A_E^t)$ modulo torsion (Proposition~\ref{mult}).

Finally, in Section~\ref{sec:class} we prove that if $C^*(E,\omega) \cong C^*(E,\omega')$, then $[\omega] = [\omega']$ for strongly connected finite directed graphs $E$ such that the only roots of unity that are eigenvalues of $A_E^t$ are the $\Pp_E$-th roots of unity (Theorem~\ref{thm:classification}). We prove this by recovering the supernatural number $[\omega]$ associated to $\omega$ from the torsion-free component of $K_0(C^*(E,\omega))$ (Theorem~\ref{omegarecovery}).

\section{Background}

\subsection{Directed graphs and their $C^*$-algebras}

We use the convention for graph $C^*$-algebras appearing in Raeburn's book
\cite{Raeburn:Graphalgebras05}. So if $E = (E^0, E^1, r, s)$ is a directed graph, then a
path in $E$ is a word $\mu = e_1\dots e_n$ in $E^1$ such that $s(e_i) = r(e_{i+1})$ for
all $i$, and we write $r(\mu) = r(e_1)$, $s(\mu) = s(e_n)$, and $|\mu| = n$. As usual, we
denote by $E^*$ the collection of paths of finite length, and $E^n := \{\mu \in E^* : |\mu| = n\}$; we also write $E^{<n} := \{\mu \in E^* : |\mu| < n\}$. We borrow the convention from the higher-rank graph literature in which we write, for example $vE^*$ for $\{\mu \in E^* : r(\mu) = v\}$, and $v E^1 w$ for $\{e \in E^1 : r(e) = v\text{ and } s(e)= w\}$. The vertex matrix of $E$ is then the $E^0 \times E^0$ integer matrix with $A_E(v,w) = |vE^1w|$.

We say that $E$ is \emph{finite} if $E^0$ is finite, that $E$ is \emph{row-finite} if $vE^1$ is finite for all $v \in E^0$, and that $E$ has no sources if each $vE^1$ is nonempty. A directed graph is \emph{strongly connected} if for every pair of vertices $v,w \in E^0$, there exists $\mu \in E^* \backslash E^0$ such that $r(\mu) = v$ and $s(\mu) = w$. The vertex matrix $A_E$ is irreducible if and only if the graph $E$ is strongly connected. The \emph{period} $\Pp_E$ of a strongly connected directed graph $E$ is given by $\Pp_E = \gcd\{|\mu| : \mu \in E^*, r(\mu) = s(\mu)\}$ (see for example \cite[Section~6]{LacaLarsenEtAl:xx14} with $k=1$). The group $\Pp_E\ZZ$ is then equal to the subgroup generated by $\{|\mu| : \mu \in v E^* v\}$ for any vertex $v$ of $E$, and so is equal to $\{|\mu| - |\nu| : \mu,\nu \in v E^* v\}$ for any $v$.

If $E$ is finite or row-finite and has no sources, then a \emph{Cuntz--Krieger $E$-family} in a $C^*$-algebra $A$ is a pair $(s,p)$, where $s = \{s_e : e \in E^1\} \subseteq A$ is a
collection of partial isometries and $p=\{p_v : v \in E^0\} \subseteq A$ is a set of mutually orthogonal projections such that $s^*_e s_e = p_{s(e)}$ for all $e \in E^1$, and $p_v = \sum_{e \in vE^1} s_e s^*_e$ for all $v \in E^0$.

The \emph{graph algebra} $C^*(E)$ is the universal $C^*$-algebra generated by a Cuntz--Krieger $E$-family \cite[Proposition~1.21]{Raeburn:Graphalgebras05}.

Theorem~7.1 of \cite{Raeburn:Graphalgebras05} says that the $K$-theory of $C^*(E)$ is given by \[ K_1(C^*(E)) \cong \ker(1-A_E^t), \quad \text{ and } \quad K_0(C^*(E)) \cong \coker(1-A_E^t).\]

\subsection{Multiplicative sequences and supernatural numbers}

A \emph{multiplicative sequence} is a sequence $\omega = (n_k)_{k=1}^\infty$ of natural numbers
with $n_k | n_{k+1}$ for all $k \in \NN$. We say that a multiplicative sequence $\omega =(n_k)_{k=1}^\infty$
divides a multiplicative sequence $\omega' = (m_j)_{j=1}^\infty$, and write $\omega|\omega'$, if for each $k \in \NN$
there exists $j(k) \in \NN$ such that $n_k|m_{j(k)}$. Define an equivalence relation $\sim$ on $\{(n_k)_{k=1}^\infty :
n_k | n_{k+1} \text{ for all } k \}$ by $\omega \sim \omega'$ if $\omega | \omega'$ and $\omega | \omega'$. The \emph{supernatural number} $[\omega]$ associated to $\omega$ is the collection $[\omega] := \{ \omega': \omega| \omega' \text{ and } \omega'|\omega \}$.

\subsection{Generalised Bunce--Deddens algebras}

Let $E = (E^0, E^1, r, s)$ be a row-finite directed graph with
no sources, and fix $n \geq 1$. Define sets
\[
E(n)^0 := E^{<n} \qquad \text{ and } \qquad E(n)^1 := \{(e,\mu) : e \in E^1, \mu \in s(e)E^{<n}\},
\] and maps
\[
s_n(e,\mu) := \mu\qquad\text{ and }\qquad
r_n(e,\mu) = \begin{cases}
        e\mu &\text{ if $|\mu| < n-1$} \\
        r(e) &\text{ if $|\mu| = n-1$.}
    \end{cases}
\]
Then $E(n) = (E(n)^0, E(n)^1, r_n, s_n)$ is a row-finite directed graph with no sources. For  $\mu \in E^*$, we write $[\mu]_n$ for the unique element of $E^{<n}$ such that $\mu = [\mu]_n \mu'$ for some $\mu'$ with $|\mu'|\in n \NN$; we think of $[\mu]_n$ as the residue of $\mu$ modulo $n$.

By Theorem~3.4 and Proposition~3.6 of \cite{RobertsonRoutSims} there exist injective homomorphisms $\tilde j_{n,mn}: C^*(E(n)) \to C^*(E(mn))$ such that \[\tilde j_{n,mn}(s_{n,(e,\mu)}) = \sum_{\tau \in s(e) E^{<mn}, [\tau]_n = \mu} s_{mn,(e,\tau)}, \quad \text{and} \quad \tilde j_{n,mn}(p_{n,\nu}) = \sum_{\tau \in E^{<mn}, [\tau]_n = \nu} p_{mn,\tau}, \] for $n,m \in \NN$ and $e \in E^1$, $\mu \in s(e) E^{<n}$ and $\nu \in E^{<n}$.

Kribs and Solel define the generalised Bunce--Deddens algebra associated to a multiplicative sequence $\omega = (n_k)_{k=1}^\infty$ by \[ C^*(E,\omega) := \varinjlim (C^*(E(n_k)), \tilde j_{n_k,n_{k+1}}).\]

\section{Applications of Perron-Frobenius theory}\label{PerronFrobenius}

In this section we analyse the invertibility of the $|E^0| \times |E^0|$ matrix $\sum_{i=0}^{(n_k/l)-1} (A_E^{il})^t$, where $l =  \gcd(\Pp_E,\omega):=\lim_{j\to\infty} \gcd(\Pp_E,n_j)$ and $k$ is such that $\gcd(\Pp_E,n_k) = l$. We also show that $\coker(1-A_E^l)^t$ is isomorphic to $l$ copies of $\coker(1-A_E^t)$. These results will be very useful when we compute the $K_1(C^*(E,\omega))$ in Section~\ref{sec:K1} and the torsion-free component of the $K_0(C^*(E,\omega))$ in Section~\ref{sec:K0}.

\begin{lem}\label{polyeigen}

For each $n \geq 1$, let $R_n$ be the polynomial over $\CC$ given by $R_n(x) = \sum_{i=0}^{n-1} x^i$. The roots of $R_n$ are the $n$-th roots of unity excluding $1$.

\end{lem}

\begin{proof}
We have $(1-x) R_n(x) = 1-x^n$, so the roots of $(1-x)R_n$ are the $n$-th roots of unity. The only root of $1-x$ is $1$, so every $n$th root of unity other than $1$ is itself a root of $R_n$. Since the degree of $R_n$ is $n-1$, these are all the roots of $R_n$.
\end{proof}

\begin{lem}\label{cmkdet}
Let $E$ be a strongly connected finite directed graph, let $\omega = (n_k)_{k=1}^\infty$ be a multiplicative sequence, and let $l= \gcd(\Pp_E, \omega)$. Then $\Pp_E / l$ and $n_k/l$ are coprime for all $k$ such that $\gcd(\Pp_E,n_k) = l$.
Hence, if the only roots of unity that are eigenvalues of $A_E^t$ are the $\Pp_E$-th
roots of unity, then $0 \not \in \sigma \big(R_{n_k/l} (A_E^{l})^t \big)$ for $k$ such that $\gcd(\Pp_E,n_k) = l$.
\end{lem}

\begin{proof}
Suppose for contradiction that $k \ge K$ and that $\Pp_E / l$ is not coprime to $n_k/l$. Say $p \not = 1$ satisfies $p|(\Pp_E / l)$ and $p|(n_k/l)$. Then $p l | \Pp_E$ and $p l | n_k$. This implies that $p l \leq l$, which is a contradiction.

For the second statement, we have \[ \sigma((A_E^{l})^t) \cap \TT = \{e^{(2 \pi i j/P_E)l}: j \in \NN \cup \{0\} \} = \{ e^{2 \pi i j /(\Pp_E/l) }: j \in \NN \cup \{0\} \}, \] by the spectral mapping theorem. By Lemma~\ref{polyeigen}, the roots of $R_{n_k/l}$ are the $n_k/l$-th roots of unity. Since $\gcd(\Pp_E / l, n_k/l)=1$, we have that $e^{2 \pi j i/(n_k/l)} \not \in \sigma((A_E^l)^t)$ for any $j \in \NN \cup \{0\}$. So $0 \not \in \sigma(R_{n_k/l}(A_E^l)^t)$.
\end{proof}

\begin{lem}\label{identitycokernel}
Let $E$ be a strongly connected finite directed graph. Then $A_E^t \delta_v + \im(1-A_E^t) = \delta_v + \im(1-A_E^t)$ for all $v \in E^0$.
\end{lem}

\begin{proof}
Fix $v \in E^0$. We have that $\delta_v - A_E^t \delta_v = (1-A_E^t)\delta_v \in \im(1-A_E^t)$, so $A_E^t \delta_v + \im(1-A_E^t) = \delta_v + \im(1-A_E^t)$.
\end{proof}

We now show that $\coker(1-A_E^l)^t \cong \bigoplus_{i=1}^l \coker(1-A_E^t)$. By \cite[Lemma~4.2]{RobertsonRoutSims} there is an equivalence relation $\sim_l$ on $E^0$ such that $v \sim_l w$ if and only if $|\lambda| \in l \ZZ$ for all $\lambda \in v E^*w$.
We enumerate the equivalence classes for $\sim_l$. Fix $v \in E^0$, and let $\Lambda_0 = [v]$. Now iteratively fix $e \in E^1$ with $r(e) \in \Lambda_i$ and let $\Lambda_{i+1} = [s(e)]$, where addition in the subscript is modulo $l$. Then $\Lambda_0, \dots, \Lambda_{l-1}$ is an enumeration of the equivalence classes in $E^0 / \sim_l$.

\begin{lem}\label{dirsuml}
Let $E$ be a strongly connected finite directed graph. Let $\omega = (n_k)_{k=1}^\infty$ be a multiplicative sequence, and let $l := \gcd(P_E, \omega)$.
There is an isomorphism  \[ \Theta: \coker(1-A_E^{l})^t \to \bigoplus_{i=0}^{l-1} \ZZ^{\Lambda_i} / (1-A_E^{l})^t \ZZ^{\Lambda_i} \] satisfying \[\Theta(\delta_v+ \im(1-A_E^l)^t) = (0, \dots, 0, \delta_v + (1-A_E^l)^t \ZZ^{\Lambda_j}, 0, \dots, 0) , \] where $v \in \Lambda_j$ for some $0 \le j \le l-1$ and $\delta_v + (1-A_E^l)^t \ZZ^{\Lambda_j}$ appears in the $j$-th position.
\end{lem}

\begin{proof}

Fix $ 0 \le j \le l-1$, and $v \in \Lambda_j$. Since $E^0 = \bigsqcup_{i=0}^{l-1} \Lambda_i$, there is an isomorphism $\theta:\ZZ^{E^0} \to \bigoplus_{i=0}^{l-1} \ZZ^{\Lambda_i}$ such that $\theta(\delta_v) = (0, \dots, 0, \delta_v, 0, \dots, 0)$, where $\delta_v$ is in the $j$-th position.

Our choice of $\Lambda_0, \dots, \Lambda_{l-1}$ ensures that $(A_E^l)^t \delta_v = \sum_{w \in E^0} |v E^l w| \delta_w \in \ZZ^{\Lambda_{j}}$ and so $(1-A_E^l)^t \delta_v \in \ZZ^{\Lambda_j}$. Hence $\theta((1-A_E^l)^t \delta_v) = (0, \dots, 0, (1-A_E^l)^t \delta_v, 0, \dots, 0) \in \bigoplus_{i=0}^{l-1} (1-A_E^{l})^t \ZZ^{\Lambda_i}$.
Therefore $\theta$ descends to an isomorphism $\Theta: \coker(1-A_E^{l})^t \to \bigoplus_{i=0}^{l-1} \ZZ^{\Lambda_i} / (1-A_E^{l})^t \ZZ^{\Lambda_i}$ satisfying the desired formula.
\end{proof}

\begin{lem}\label{isovarup}

Let $E$ be a strongly connected finite directed graph. Let $\omega = (n_k)_{k=1}^\infty$ be a multiplicative sequence, and let $l = \gcd(P_E, \omega)$. For each $0 \le j \le l -1$, there is an isomorphism $\Phi_j: \ZZ^{\Lambda_j} / (1-A_E^l)^t \ZZ^{\Lambda_j} \to \ZZ^{E^0} / \im(1-A_E)^t$ satisfying \[ \Phi_j(\delta_v+(1-A_E^l)^t \ZZ^{\Lambda_j}) = \delta_v+(1-A_E^t) \ZZ^{E^0}, \] for some $v \in \Lambda_j$.

\end{lem}

\begin{proof}
Fix $0 \le j \le l-1$. The formula $(1-A_E^l)^t = (1-A_E^t) \big(\sum_{i=0}^{l-1} (A_E^i)^t \big)$ shows that $\im(1-A_E^l)^t \subseteq \im(1-A_E^t)$. Since $(1-A_E^l)^t \ZZ^{\Lambda_j} \subseteq \im(1-A_E^l)^t$, it follows that the map $\ZZ^{\Lambda_j} \to \ZZ^{E^0}$ given by $\delta_v \mapsto \delta_v$ for $v \in  \Lambda_j$, descends to a homomorphism $\Phi_j: \ZZ^{\Lambda_j} / (1-A_E^l)^t \ZZ^{\Lambda_j} \to \ZZ^{E^0} / \im(1-A_E)^t$ satisfying $\Phi_j (\delta_v + (1-A_E^l)^t \ZZ^{\Lambda^j}) = \delta_v+\im(1-A_E^t),$ for $v \in \Lambda_j$.

We must show that $\Phi_j$ is an isomorphism. To see that $\Phi_j$ is surjective, fix $0 \le k \le l-1$ and $v \in \Lambda_k$. Then $(A_E^{j-k})^t \delta_v \in \ZZ^{\Lambda_j}$ and \[ \delta_v + \im(1-A_E^t) = (A_E^{j-k})^t \delta_v + \im(1-A_E^t) = \Phi_j\big((A_E^{j-k})^t \delta_v + (1-A_E^l)^t \ZZ^{\Lambda_j} \big). \]

To see that $\Phi_j$ is injective, fix $a = \sum_{v \in \Lambda_j} a_v \delta_v \in \ZZ^{\Lambda_j}$ such that $\Phi_j(a + (1-A_E^l)^t \ZZ^{\Lambda_j}) = 0$. That is, $a \in \im(1-A_E^t)$. Say $a = (1-A_E^t)b$ where $b = \sum_{w \in E^0} b_w \delta_w$. Let $b_k := b|_{\Lambda_k} = \sum_{w \in \Lambda_k} b_w \delta_w$ for each $0 \le k \le l-1$. Since $a\in \ZZ^{\Lambda_j}$, we have $0 = a|_{\Lambda_k} = ((1-A_E^t)b)|_{\Lambda_k} = b_k- A_E^t b_{k-1}$, for all $0 \le k \le l-1$, $k \not = j$, where subtraction in the subscript is modulo $l$. Therefore $b_k = (A_E^t)^{k-j} b_j$ for each $0 \le k \le l-1$, $k \not = j$, where subtraction in the superscript is modulo $l$. Hence \begin{align*} a = (1-A_E^t)b &= (1-A_E^t)(b_0 + \dots + b_{l-1}) \\ &= (1-A_E^t) \Big(\sum_{k=0}^{l-1} (A_E^k)^t \Big) b_j = (1-A_E^l)^t b_j \in (1-A_E^l)^t \ZZ^{\Lambda_j}. \qedhere \end{align*} \end{proof}

\begin{cor}\label{C*(E(l))1}
Let $E$ be a strongly connected finite directed graph. Let $\omega = (n_k)_{k=1}^\infty$ be a multiplicative sequence, and let $l = \gcd(P_E, \omega)$. There is an isomorphism $\rho: \coker(1-A_E^l)^t  \to \bigoplus_{i=1}^l \coker(1-A_E^t)$
satisfying \[ \rho \big( \delta_v+\im(1-A_E^l)^t \big) = \big(0, \dots, 0, \delta_v+\im(1-A_E^t), 0, \dots, 0 \big), \] where $v \in \Lambda_j$ for some $0 \le j \le l-1$ and $\delta_v + \im(1-A_E^t)$ appears in the $j$-th position.
\end{cor}

\begin{proof}
Define $\rho := \big(\bigoplus_{i=0}^{l-1} \Phi_i \big) \circ \Theta$. It follows from Lemma~\ref{isovarup} and Lemma~\ref{C*(E(l))1} that $\rho$ is an isomorphism that satisfies the desired formula.
\end{proof}

\section{Computing $K_1(C^*(E,\omega))$}\label{sec:K1}

In this section we compute $K_1(C^*(E,\omega))$ where $E$ is a strongly connected finite graph $E$ such that the only roots of unity that are eigenvalues of $A_E^t$ are the $\Pp_E$-th roots of unity, and $\omega$ is a multiplicative sequence. The main result of this section is the following.

\begin{thm}\label{K_1 limit}
Let $E$ be a strongly connected finite graph and suppose that the only roots of unity that are eigenvalues of $A_E^t$ are the $\Pp_E$-th roots of unity. Let $\omega = (n_k)_{k=1}^\infty$ be a multiplicative sequence and let $l := \gcd(\Pp_E,\omega)$. Then
\[ K_1\big(C^*(E,\omega) \big) =\bigoplus_{i=1}^l \ker(1-A_E^t).\]
\end{thm}

To prove Theorem~\ref{K_1 limit} we need a series of results. We begin by studying $\ker(1-A_{E(n)}^t)$ for $n \geq 1$.

Let $\{ \delta_{v}: v \in E^0 \}$ be the generators of $\ZZ^{E^0}$ and let $\{ \delta_{\mu,n}: \mu \in E^{<n} \}$ be the generators of $\ZZ^{E^{<n}}$. For $0 \le k \le n-1$ and $a = \sum_{\mu \in E^{<n}} a_\mu \delta_{\mu,n} \in \ZZ^{E^{<n}}$,
we define $a_k := \sum_{\mu \in E^k} a_\mu \delta_{\mu,n} \in \ZZ^{E^{<n}}$ and $a|_{\ZZ^{E^k}}:=\sum_{\mu \in E^k} a_\mu \delta_{\mu,k} \in \ZZ^{E^k}$. For $b = \sum_{v \in E^0} b_v \delta_v \in \ZZ^{E^0}$, we define $\iota_n(b) := \sum_{v \in E^0} b_v \delta_{v,n} \in \ZZ^{E^{<n}}$.

\begin{lem}\label{lem:keriso}

Let $E$ be a row-finite directed graph with no sources and let $n \geq 1$. There is an isomorphism $\psi_n: \ker(1 - A_{E(n)}^t) \to \ker(1-A_E^n)^t$ satisfying $\psi_n(a) = a|_{E^0}$ for $a \in \ker(1-A_{E(n)}^t)$.

\end{lem}

\begin{proof}

Define $\psi_n: \ker(1-A_{E(n)}^t) \to \ZZ^{E^0}$ by $\psi_n(a) = a|_{\ZZ^{E^0}}$
for $a \in (1-A_{E(n)}^t)$. We check that $\psi_n(\ker(1-A_{E(n)}^t) \subseteq \ker(1-A_E^n)^t$. Let $a \in \ker(1-A_{E(n)}^t)$. Then \[ (1-A_E^n)^t (\psi_n(a))= (1-A_E^n)^t (a|_{\ZZ^{E^0}}) = ( (1-A_{E(n)}^n)^ta_0 )|_{\ZZ^{E^0}} = 0. \] So $\psi_n(a) \in \ker(1-A_E^n)^t$, and hence $\psi_n$ descends to a homomorphism $\ker(1-A_{E(n)}^t) \to \ker(1-A_E^n)^t$ which we also label $\psi_n$.

Define $\varphi_n: \ker(1-A_{E}^n)^t \to \ZZ^{E^{<n}}$ by $\varphi_n(b) =
\sum_{i=0}^{n-1} (A_{E(n)}^t)^i (\iota_n(b))$ for $b \in \ker(1-A_{E}^n)^t$. We check
$\varphi_n(\ker(1-A_E^n)^t) \subseteq \ker(1-A_{E(n)}^t)$. Let $b \in \ker(1-A_E^n)^t$.
Then \begin{align*} (1-A_{E(n)}^t) (\varphi_n(b)) &= (1-A_{E(n)}^t) \Big(\sum_{i=0}^{n-1} (A_{E(n)}^i)^t (\iota_n(b)) \Big) \\ &= (1-A_{E(n)}^n)^t (\iota_n(b)) \\ &= \iota_n \big( (1 - A_{E(n)}^n)^t (b) \big)= 0. \end{align*} So $\varphi_n(b) \in \ker(1-A_{E(n)}^t)$, and hence $\varphi_n$ descends to a homomorphism $\ker(1-A_E^n)^t \to \ker(1-A_{E(n)}^t)$ which we also label $\varphi_n$.

We check that $\varphi_n$ is an inverse for $\psi_n$. Let $a \in \ker (1-A_{E(n)}^t)$.
Fix $k <n$. We have \[ 0 = (1 - A_{E(n)}^t) (a_k) = \begin{cases} a_k - A_{E(n)}^t (a_{k+1})
& \text{if } k \not = n - 1 \\ a|_{E^{n-1}} - A_{E(n)}^t (a_0) & \text{if } k = n -1. \end{cases} \]

So $a_{n-1} = A_{E(n)}^t (a_0)$. Then $a_{n-2} = A_{E(n)}^t (a_{n-1}) =
(A_{E(n)}^2)^t (a_0)$. Repeating this step yields $a_{n-i} = (A_{E(n)}^i)^t (a_0)$
for $i < n$. Since $a_0 = \iota_n(a|_{\ZZ^{E^0}})$, we have $\varphi_n(\psi_n(a)) =
\varphi_n(a|_{\ZZ^{E^0}}) = \sum_{i=0}^{n-1} (A_{E(n)}^i)^t a_0 = a$.

Now, we check that $\psi_n$ is an inverse for $\varphi_n$. Let $v \in E^0$ and $0 \le i < n$. Repeated applications of \eqref{eq:matmult} shows that $\big(A_{E(n)}^i \big)^t \delta_{v,n} \in \lsp\{ \delta_{\mu,n}: \mu \in E^{n-i} \}$. Thus \begin{align}\label{eq:mat cases}
\big( \big(A_{E(n)}^i \big)^t \delta_{v,n} \big) \big|_{\ZZ^{E^0}} = \begin{cases} \delta_{v} & \text{ if } i = 0 \\ 0 & \text{ otherwise}. \end{cases} \end{align}

Now, let $b \in \ker(1-A_E^n)^t$.  By \eqref{eq:mat cases}, we have \[ \psi_n (\varphi_n(b)) =
\Big(\sum_{i=0}^{n-1} (A_{E(n)}^i)^t (\iota_n(b)) \Big) \Big|_{\ZZ^{E^0}} = b. \qedhere \] \end{proof}

Suppose $E$ is a row-finite directed graph with no sources. Define the skew-product graph $E \times_1 \ZZ$ as the graph with edge set $(E \times_1 \ZZ)^1 = E^1 \times \ZZ$ and vertex set $(E \times_1 \ZZ)^0 = E^0 \times \ZZ$ and range and source maps defined by \[ r(e, k) = (r(e), k-1) \text{ and } s(e,k) = (s(e),k). \] For each $n \geq 1$, we denote by $s_{n,((e,\mu),k)}$ and $p_{n,(\mu,k)}$ the generators of $C^*(E(n) \times_1 \ZZ)$. Proposition~6.7 of \cite{Raeburn:Graphalgebras05} gives a natural action $\beta_{E(n)}$ of $\ZZ$ on $C^*(E(n) \times_1 \ZZ)$ such that $(\beta_{E(n)})_m(s_{n,((e,\mu),k)}) = s_{n,((e,\mu),k+l)}$. By \cite[Lemma~7.10]{Raeburn:Graphalgebras05} there is an isomorphism $\phi_{E(n)}$ of $C^*(E \times_1 \ZZ)$ onto the crossed product $C^*(E(n)) \rtimes \TT$ such that $\phi_{E(n)} \circ (\beta_{E(n)})_m = \hat \gamma^n_m \circ \phi_{E(n)}$, where $\hat \gamma^n$ is the dual of the gauge action $\gamma^n$ of $C^*(E(n))$.

\begin{lem}
Let $E$ be a row-finite directed graph with no sources, and let $n,m \in \NN$.
There is a homomorphism $i_{n,mn}: C^*(E(n) \times_1 \ZZ) \to C^*(E(mn) \times_1 \ZZ)$
such that \begin{align*}&  i_{n,mn}(s_{n,((e,\mu),1)}) = \sum_{\tau \in s(e) E^{<mn}, [\tau]_n = \mu} s_{mn,((e,\tau),1)} \quad \text{ and } \\ & i_{n,mn}(p_{n,(\mu,1)}) = \sum_{\tau \in E^{<mn}, [\tau]_n = \mu} p_{mn,(\tau,1)},\end{align*} for all $n \geq 1$.
\end{lem}

\begin{proof} Let $(i_{C^*(E(n))},i_\TT)$ be the universal covariant representation of $(C^*(E(n)), \TT,\gamma^n)$. Recall the injective homomorphism $\tilde j_{n,mn}: C^*(E(n)) \to C^*(E(mn))$. We show that $\tilde j_{n,mn}$ is $\TT$-equivariant. For $e \in E^1$ and $\mu \in s(e) E^{<n}$ and $z \in \TT$, we have
\begin{align*} \tilde j_{n,mn}(\gamma^n_z(s_{n,(e,\mu)}) &= \tilde j_{n,mn}(z s_{n,(e,\mu)}) \\ &= \sum_{\tau \in E^{<mn}, [\tau]_n=\mu} z s_{mn,(e,\tau)} = \gamma^{mn}_z (\tilde j_{n,mn}(s_{n,(e,\mu)})),\end{align*} and similarly for $\mu \in E^{<n}$, $\tilde j_{n,mn}(\gamma^n_z(p_{n,\mu})) = (\gamma^{mn}_z(\tilde j_{n,mn}(p_{n,\mu}))$.

By \cite[Corollary~2.48]{Williams:Crossedproducts} there is a homomorphism $\tilde j_{n,mn} \times 1: C^*(E(n)) \rtimes \TT \to C^*(E(mn)) \rtimes \TT$ satisfying \[(\tilde j_{n,mn} \times 1)(i_{C^*(E(n))}(a)i_\TT(z))=i_{C^*(E(mn))}(\tilde j_{n,mn}(a))i_\TT(z)\]
for all $a \in C^*(E(n))$ and $z \in \TT$.

Define $i_{n,mn} := \phi_{E(mn)}^{-1} \circ (\tilde j_{n,mn} \times 1) \circ \phi_{E(n)}$. Let $((e,\mu),1) \in E(n)^1 \times_1 \ZZ$ and let $f_1(z) = z$ for $z \in \TT$. We calculate \begin{align*} \big( \phi_{E(mn)}^{-1} \circ (\tilde j_{n,mn} \times 1) \circ \phi_{E(n)} \big) (s_{n,((e,\mu),1)}) &= \phi_{E(n)}^{-1}\big((\tilde j_{n,mn} \times 1) (i_A(s_{(n,(e,\mu))}) i_\TT(f_1))\big) \\ &= \phi_{E(n)}^{-1} \Big(i_A \Big( \sum_{\tau \in s(e) E^{<mn}, [\tau]_n = \mu} s_{mn,(e,\tau)} \Big) i_\TT (f_1) \Big) \\ &= \sum_{\tau \in s(e) E^{<mn}, [\tau]_n = \mu} s_{mn,((e,\tau),1)}. \end{align*}

Similarly, for $(\mu,1) \in E(n)^0 \times_1 \ZZ$, we have \begin{align*} \big( \phi_{E(mn)}^{-1} \circ (\tilde j_{n,mn} \times \id) \circ \phi_{E(n)} \big) (p_{n,(\mu,1)}) &= \phi_{E(n)}^{-1}\big((\tilde j_{n,mn} \times \id) (i_A(p_{n,\mu}) i_\TT(f_1))\big) \\ &= \phi_{E(n)}^{-1} \Big(i_A \Big( \sum_{\tau \in E^{<mn}, [\tau]_n = \mu} p_{mn,\tau} \Big) i_\TT (f_1) \Big) \\ &= \sum_{\tau \in E^{<mn}, [\tau]_n = \mu} p_{mn,(\tau,1)}. \qedhere  \end{align*} \end{proof}

\begin{prp}\label{K_1 diagram}
Let $E$ be a row-finite directed graph with no sources and let $n,m \in \NN$. There are isomorphisms $K_1(C^*(E(n))) \to \ker(1-A_E^n)^t$ and $K_1(C^*(E(mn))) \to \ker(1-A_E^{mn})^t$ such that the following diagram commutes.

\begin{center}
\begin{tikzcd}
K_1(C^*(E(n))) \arrow{d}{} \arrow{r}{K_1(\tilde j_{n,mn})}
& K_1(C^*(E(mn))) \arrow{d}{} \\
\ker(1-A_E^n)^t \arrow{r}{x \mapsto x}
& \ker(1-A_E^{mn})^t
\end{tikzcd}
\end{center}

\end{prp}

\begin{proof}

The naturality of the Pimsner--Voiculescu diagram gives the following commutative
diagram (see \cite[Lemma~7.12]{Raeburn:Graphalgebras05}).

\begin{center}
\begin{tikzcd}
K_1(C^*(E(n))) \arrow{d}{} \arrow{r}{K_1(\tilde j_{n,mn})}
& K_1(C^*(E(mn))) \arrow{d} \\
\ker(1-(\beta_{E(n)})_*^{-1}) \arrow{r}{}
& \ker(1-(\beta_{E(mn)})^{-1}_*)
\end{tikzcd}
\end{center}

By \cite[Lemma~7.13]{Raeburn:Graphalgebras05} there is an injection $\sigma_n:
\ZZ^{E^{<n}} \to K_0(C^*(E(n) \times_1 \ZZ))$ satisfying $\sigma_n(\delta_{\mu,n})
= [p_{n,(\mu,1)}]_0$. Define $\phi_{n,mn}: \ZZ^{E^{<n}} \to \ZZ^{E^{<mn}}$ by
$\phi_{n,mn}(\delta_{\mu,n}) = \sum_{\tau \in E^{<mn}, [\tau]_n =\mu} \delta_{\tau,mn}$
for $\mu \in E^{<n}$. We claim that the following diagram commutes.

\begin{equation}\label{diagram:torestrict}
\begin{tikzcd}
K_0(C^*(E(n) \times_1 \ZZ)) \arrow{r}{K_0(i_{n,mn})}
& K_0(C^*(E(mn) \times_1 \ZZ)) \\
\ZZ^{E^{<n}} \arrow{u}{\sigma_n}  \arrow{r}{\phi_{n,mn}}
& \ZZ^{E^{<mn}} \arrow{u}{\sigma_{mn}}
\end{tikzcd}
\end{equation}

To prove this claim, fix $\mu \in E^{<n}$. Then \begin{align*} (\sigma_{mn}^{-1} \circ K_0(i_{n,mn}) \circ \sigma_n)(\delta_{\mu,n}) &= (\sigma_{mn}^{-1} \circ K_0(i_{n,mn}))([p_{(n,\mu,1)}]_0) \\ &= \sigma_{mn}^{-1}([i_{n,mn}(p_{(n,\mu,1)}]_0) \\
&= \sigma_{mn}^{-1} \Big( \sum_{\tau \in E^{<mn}, [\tau]_n = \mu} [p_{(mn,\tau,1)} ]_0 \Big) \\ &= \sum_{\tau \in E^{<mn}, [\tau]_n = \mu} \delta_{\tau,mn} \\ &= \phi_{n,mn}(\delta_{\mu,n}). \end{align*}

It follows from \cite[Theorem~7.16]{Raeburn:Graphalgebras05} that $\sigma_n$ restricts to an isomorphism of $\ker(1-A_{E(n)}^t)$ onto $\ker(1-(\beta_{E(n)})_*^{-1})$. Restricting diagram~\ref{diagram:torestrict} to the subgroups $\ker(1-(\beta_{E(n)})^{-1}_*) \subseteq K_0(C^*(E(n)) \times_1 \ZZ))$ and $\ker(1-A_{E(n)}^t) \subseteq \ZZ^{E^{<n}}$ yields the following commuting diagram.

\begin{center}
\begin{tikzcd}
\ker(1-(\beta_{E(n)})_*^{-1}) \arrow{d}{} \arrow{r}{K_0(i_{n,mn})}
& \ker(1-(\beta_{E(mn)})^{-1}_*) \arrow{d}{} \\
\ker(1-A_{E(n)}^t) \arrow{r}{\phi_{n,mn}}
&\ker(1-A_{E(mn)}^t)
\end{tikzcd}
\end{center}

Now, we claim that the following diagram commutes.

\begin{center}
\begin{tikzcd}
\ker(1-A_{E(n)}^t) \arrow{d}{\psi_n} \arrow{r}{\phi_{n,mn}}
&\ker(1-A_{E(mn)}^t) \arrow{d}{\psi_{mn}} \\
\ker(1-A_E^n)^t \arrow{r}{x \mapsto x}
& \ker(1-A_E^{mn})^t
\end{tikzcd}
\end{center}

To prove this claim, fix $x \in \ker(1-A_{E(n)}^t)$. Then \[ \psi_{mn}(\phi_{n,mn}(x))
= \psi_{mn} \Big(\sum_{\mu \in E^{<n}} x_\mu \sum_{\tau \in E^{<mn},
[\tau]_n = \mu} \delta_{\tau,mn} \Big) = \sum_{v \in E^0} x_v \delta_{v} = \psi_n(x). \]

Combining the preceding commutative diagrams gives the desired commutative diagram.
\end{proof}

\begin{proof}[Proof of Theorem~\ref{K_1 limit}]
By \cite[Theorem~6.3.2]{Rordam2000}, we have \[K_1\big(C^*(E,\omega) \big) \cong \varinjlim \big(K_1 (C^*(E(n_k)),K_1(j_{n_k,n_{k+1}}) \big). \] By Proposition~\ref{K_1 diagram}, we have \[ \big( \varinjlim K_1 (C^*(E(n_k)),K_1(j_{n_k,n_{k+1}}) \big) \cong \varinjlim \big(\ker(1-A_E^{n_k})^t, x \mapsto x\big). \]

By Lemma~\ref{cmkdet} the matrix $\sum_{j=0}^{n_k/l-1} (A_E^{j l})^t$ is invertible for $k$ such that $\gcd(\Pp_E,n_k) = \gcd(\Pp_E,\omega)$. So \[ \ker(1-A_E^{n_k})^t = \ker\Big( \Big( \sum_{j=0}^{n_k/l-1} (A_E^{j l})^t \Big) (1-A_E^l)^t \Big) = \ker(1-A_E^l)^t,\] for $k$ such that $\gcd(\Pp_E,n_k) = \gcd(\Pp_E,\omega)$. Hence \[ \big(\ker(1-A_E^{n_k})^t, x \mapsto x\big)
\cong \ker(1-A_E^l)^t.\] Combining the previous three isomorphisms gives an isomorphism
\[ K_1(C^*(E,\omega)) \cong \ker (1-A_E^l)^t.\] Now, $\ker(1-A_E^l)^t \cong \ZZ^r$, where $r = \rank \coker(1-A_E^l)^t = l \cdot \rank \coker(1-A_E^t)$ by Corollary~\ref{C*(E(l))1}. So $\ker (1-A_E^l)^t \cong \bigoplus_{i=1}^l \ker(1-A_E^l)^t$, giving the result.
\end{proof}

\section{Computing the torsion-free component of $K_0(C^*(E,\omega))$}\label{sec:K0}

In this section we calculate the torsion-free component of $K_0(C^*(E,\omega))$. We will use this group in Section~\ref{sec:class} to recover the supernatural number $[\omega]$ associated to $\omega$. In order to state the main theorem of this section, we need the following lemma.

\begin{lem}\label{adjoingroup}
Let $A$ be a free abelian group and let $\omega = (n_k)_{k=1}^\infty$ be a multiplicative sequence. Define an equivalence relation $\sim$ on $A \times \NN$, by $(a,j) \sim (a',j')$ if \[\frac{\max\{n_j, n_{j'} \}}{n_j} a = \frac{\max \{n_j, n_{j'} \}}{n_{j'}}a',\] and define \[ A \Big[ \frac{1}{\omega} \Big] := \left\{(a,j): a \in A, j \in \NN \right\} / \sim. \] Then $A \big[ \frac{1}{\omega} \big]$ is a torsion-free abelian group under the operation \[ [(a,j)] + [(a',j')] = \begin{cases} [((n_{j'}/n_j) \cdot a + a',j')] \quad & \text{ if } j' \ge j \\ [(a + (n_j/n_{j'}) \cdot a', j)] \quad & \text{ if } j \ge j'. \end{cases} \] Moreover, $\rank A \left[ \frac{1}{\omega} \right] = \rank A.$
\end{lem}

\begin{proof}

Closure, associativity, and commutativity follow easily since $A$ is abelian. Let $0$ be the identity element of $A$. Then $[(0,i)] + [(a,i)] = [(0 + a,i)] = [(a,i)]$ so $[(0,i)]$ is an identity for $A \big[ \frac{1}{\omega} \big]$. Fix $a \in A$ and let $-a$ be the inverse. Then $[(a,i)] + [(-a,i)] = [(a-a,i)] = [(0,i)]$, so $[(-a, i)]$ is an inverse for $[(a,i)]$.

If $k \cdot [(a,i)] = [(0,i)]$, then $[(k \cdot a,i)] = [(0,i)]$, so $k \cdot a = 0$ forcing $a = 0$ since $A$ is free abelian. To see that $\rank A \left[ \frac{1}{\omega} \right] = \rank A$, let $\{ a_\alpha \}$ be a maximal linearly independent subset of $A$. Suppose $[(0,i)] = \sum_\alpha c_\alpha \cdot [(a_\alpha, i)]$ for $c_\alpha \in \NN$ with all but finitely many nonzero. Then $[(0,i)] = \sum_\alpha [(c_\alpha \cdot a_\alpha, i)] = [(\sum_\alpha c_\alpha \cdot a_\alpha, i)]$, so $0 = \sum_\alpha c_\alpha \cdot a_\alpha$, and since $\{a_i\}$ is linearly independent, $c_\alpha = 0$ for all $\alpha$. Hence $\{ [(a_\alpha, i)]\}$ is a linearly independent subgroup of $A \left[ \frac{1}{\omega} \right].$ To see that it is maximal, take $c \in \NN$ and $b \in A$. Then $\sum_\alpha c_\alpha [(a_\alpha,i)] + c [(b,i)] = [(\sum_\alpha c_\alpha \cdot a_\alpha + c\cdot b, i)] = [(0,i)]$, by the maximality of $\{a_\alpha\}$.
\end{proof}

\begin{rmk}
We have $A\Big[\frac{1}{\omega}\Big] \cong A \bigotimes \ZZ \Big[\frac{1}{\omega} \Big]$ via the map $[a,j] \to a \bigotimes \frac{1}{n_j}$. We will regard the elements $[(a, j)]$ as formal fractions and write $a/ n_j$ for $[(a, j)]$. 
\end{rmk}

We now state the main theorem of this section about the torsion-free component of $K_0(C^*(E,\omega))$. Recall that the torsion subgroup of an abelian group $A$ consists of the nonzero elements of $A$ which have finite order.

\begin{thm}\label{supchar}
Let $E$ be a strongly connected finite directed graph. Let $\Pp_E$ denote the period of $E$, and let $l = \gcd(\Pp_E,\omega)$. Suppose $1$ is an eigenvalue of $A_E^t$ and that the only roots of unity that are eigenvalues of $A_E^t$ are the $\Pp_E$-th roots of unity. Let $\omega = (n_k)_{k=1}^\infty$ be a multiplicative sequence. Let $\tor_E$ denote the torsion subgroup of $K_0(C^*(E))$, and $tor_{(E,\omega)}$ the torsion subgroup of $K_0(C^*(E,\omega))$. There is an isomorphism \[ \Psi:K_0(C^*(E,\omega))/ \tor_{(E,\omega)} \to \bigoplus_{i=1}^l \big(K_0(C^*(E)) / \tor_{E} \big) \Big[ \frac{1} {\omega } \Big] \] satisfying \[ \Psi([1_{C^*(E,\omega)}]_0 + \tor_{E,\omega}) = ([1_{C^*(E)}]_0 + \tor_E, \dots, [1_{C^*(E)}]_0 + \tor_{E}). \]
\end{thm}

To prove Theorem~\ref{supchar} we need a series of lemmas. We begin by studying $K_0(C^*(E(n)) \cong \coker(1- A_{E(n)}^t)$ for $n \geq 1$.

\begin{lem}\label{mu s(mu)}
Let $E$ be a row-finite directed graph with no sources and let $n \geq 1$. Then
\begin{align}\label{eq:matmult} A_{E(n)}^t \delta_{\mu,n} =
\begin{cases} \delta_{\mu_2 \dots \mu_{|\mu|},n}
& \text{ if } \mu \in E^{<n} \backslash E^0 \\
\sum_{\lambda \in \mu E^n} \delta_{\lambda_2 \dots \lambda_n,n} & \text{ if } \mu \in E^0. \end{cases} \end{align} Moreover, $\delta_{\mu,n} - \delta_{s(\mu),n} \in \im(1-A_{E(n)}^t)$ for each $\mu \in E^{<n}$.
\end{lem}

\begin{proof}

 Let $\mu \in E^n \setminus E^0$. We calculate
\begin{align*}
A_{E(n)}^t \delta_{\mu,n} &= \sum_{\nu \in E^{<n}} A_{E(n)}^t
(\nu,\mu) \delta_{\nu,n} = \sum_{\nu \in E^{<n}} |\mu E(n)^1 \nu| \delta_{\nu,n} \\
&= \sum_{\nu \in E^{<n}, e \in E^1 r(\nu), [e \nu]_n = \mu} \delta_{\nu,n} = \delta_{\mu_2 \dots \mu_{|\mu|},n}.
\end{align*}

Let $\mu \in E^0$. Then \[ A_{E(n)}^t \delta_{\mu,n} = \sum_{\nu \in E^{<n}} A_{E(n)}^t(\nu,\mu) \delta_{\nu,n}
= \sum_{\nu \in E^{<n}} |\mu E(n)^1 \nu| \delta_{\nu,n} = \sum_{\lambda \in \mu E^n} \delta_{\lambda_2 \dots \lambda_n,n}. \]

The final statement clearly holds when $\mu \in E^0$, so let $\mu \in E^{<n} \backslash E^0$. Repeated applications of the first case of \eqref{eq:matmult} give $( A_{E(n)}^{|\mu|})^t \delta_{\mu,n} = \delta_{s(\mu),n},$ so $\delta_{\mu,n} - \delta_{s(\mu),n} = (1- A_{E(n)}^{|\mu|})^t \delta_{\mu,n} \in \im (1- A_{E(n)}^t)$.
\end{proof}

\begin{lem}\label{cokeriso}

Let $E$ be a row-finite directed graph with no sources and let $n \geq 1$. There is an isomorphism $\psi_n: \coker(1 - A_E^n)^t \to \coker(1 - A_{E(n)}^t)$ satisfying $\psi_n (\delta_{v} + \im(1 - A_E^n)^t) = \delta_{v,n} + \im(1-A_{E(n)}^t)$ for $v \in E^0$.

\end{lem}

\begin{proof}

Define a map $\psi_n: \ZZ^{E^0} \to \coker(1 - A_{E(n)}^t)$ by $\psi_n(\delta_{v}) =
\delta_{v,n} + \im(1- A_{E(n)}^t)$. We show that $\psi_n(\im(1-A_E^n)^t) \subseteq \im(1-A_{E(n)}^t)$. Let $v \in E^0$. Repeated applications of \eqref{eq:matmult} give $(A_{E(n)}^n)^t \delta_{v,n} = \sum_{\lambda \in v E^n} \delta_{s(\lambda),n} = \sum_{w \in E^0} |v E^n w| \delta_{w,n} = \sum_{w \in E^0} (A_E^n)^t (w,v) \delta_w = \psi_n \big( (A_E^n)^t \delta_v \big)$, so \[ \psi_n \big((1- A_E^n)^t \delta_{v} \big) = (1- A_{E(n)}^n)^t \delta_{v,n} + \im(1-A_{E(n)}^n)^t = \im(1 - A_{E(n)}^n)^t \subseteq \im(1 - A_{E(n)}^t). \] Thus $\psi_n$ descends to a homomorphism $\coker(1 - A_E^n)^t \to \coker(1 - A_{E(n)}^t),$ which we also label by $\psi_n$, satisfying $\psi_n (\delta_{v} + \im(1 - A_E^n)^t) = \delta_{v,n} + \im(1-A_{E(n)}^t)$ for $v \in E^0$.

Define a map $\varphi_n: \ZZ^{E^{<n}} \to \coker(1 - A_E^n)^t$ by $\varphi_n(\delta_{\mu,n}) = \delta_{s(\mu)} + \im(1-A_E^n)^t$.
We show that $\varphi_n(\im(1 - A_{E(n)}^t) = \im(1-A_E^n)^t$. Take $(1 - A_{E(n)}^t) \delta_{\mu,n} \in \im(1 - A_{E(n)}^t)$. If $\mu \in E^{<n} \backslash E^0$, then \[ \varphi_n((1-A_{E(n)}^t) \delta_{\mu,n}) = \varphi_n(\delta_{\mu,n} - \delta_{\mu_2 \dots \mu_{|\mu|},n}) = \delta_{s(\mu)} - \delta_{s(\mu)} + \im(1-A_E^n)^t = \im(1-A_E^n)^t, \] by the first case of \eqref{eq:matmult}. If $\mu \in E^0$, then applying the second case of \eqref{eq:matmult} at the first equality, we have
\begin{align*} \varphi_n \big(\delta_{v,n} - A_{E(n)}^t \delta_{v,n} \big)
&= \varphi_n \big(\delta_{v,n} - \sum_{\lambda \in v E^n} \delta_{\lambda_2 \dots \lambda_n,n} \big) \\
&= \delta_{v} - \sum_{\lambda \in v E^n} \delta_{s(\lambda)} + \im(1-A_E^n)^t \\
&= \delta_{v} - \sum_{w \in E^0} |v E^n w| \delta_{w} + \im(1-A_E^n)^t \\
&= \delta_{v} - \sum_{w \in E^0} (A_E^n)^t(w,v) \delta_{w} + \im(1-A_E^n)^t \\
&= (1 - A_E^n)^t \delta_{v} + \im(1-A_E^n)^t \\
&= \im(1-A_E^n)^t. \end{align*}
Thus $\varphi_n$ descends to a homomorphism $\coker(1- A_{E(n)}^t) \to \coker(1 - A_E^n)^t$, which we also label $\varphi_n$.

To show that $\psi_n$ is an isomorphism, we show that $\psi_n$ and $\varphi_n$ are mutually inverse. Let $\mu \in E^{<n}$. Then \begin{align*} \psi_n(\varphi_n(\delta_{\mu,n} + \im(1 - A_{E(n)}^t))) &=
\varphi_n(\delta_{s(\mu)}+\im(1 - A_E^n)^t) \\ &= \delta_{s(\mu),n}+\im(1 - A_{E(n)}^t) =
\delta_{\mu,n} + \im(1-A_{E(n)}^t) \end{align*} by Lemma~\ref{mu s(mu)}, so $\psi_n \circ \varphi_n$ is the identity on $\coker(1-A_{E(n)}^t)$. Now, let $v \in E^0$. Then \[ \varphi_n(\psi_n (\delta_{v}+\im(1 - A_E^n)^t)) = \varphi_n(\delta_{v,n} +\im(1 - A_{E(n)}^t)) = \delta_{v}+\im(1 - A_E^n)^t, \] so $\varphi_n \circ \psi_n$ is the identity on $\coker(1-A_E^n)^t$.
\end{proof}

\begin{rmk}\label{raeburnsktheory} For each $n \geq 1$, let $\sigma_n: \coker(1 - A_{E(n)}^t) \to K_0(C^*(E(n)))$ be the isomorphism of \cite[Theorem~7.16]{Raeburn:Graphalgebras05}. Looking into the proof of \cite[Theorem~7.1]
{Raeburn:Graphalgebras05} shows that this isomorphism is given by $\sigma_n(\delta_{\mu,n} + \im(1 - A_{E(n)}^t)) = [p_{\mu,n}]_0$ for $\mu \in E^{<n}$. So $\sigma_n \circ \psi_n: \coker (1 - A_E^n)^t \to K_0(C^*(E(n)))$ is an isomorphism
satisfying $(\sigma_n \circ \psi_n)\big(\delta_{v} + \im(1-A_E^n)^t\big) = [p_{v,n}]_0$ for $v \in E^0$. By Lemma~\ref{mu s(mu)}, we have $[p_{\mu,n}]_0 - [p_{v,n}]_0 = \sigma_{n} (\delta_{\mu,n} - \delta_{v,n} + \im(1 - A_{E(n)}^t)) = 0$ for any $\mu \in E^{<n}v$. So $(\sigma_n \circ \psi_n)\big(\delta_{v} + \im(1-A_E^n)^t\big) = [p_{\mu,n}]_0$ for any $\mu \in E^{<n}v$. \end{rmk}

\begin{lem}\label{comm}
Let $E$ be a row-finite directed graph with no sources and let $n, m \in \NN$. The following diagram commutes.

\begin{center}

\begin{tikzcd}
\ZZ^{E^0} \arrow{d}{} \arrow{r}{\sum_{i = 0}^{m-1} (A_E^{i n})^t}
& \ZZ^{E^0} \arrow{d}{} \\
\coker(1-A_E^n)^t \arrow{d}{\sigma_n \circ \psi_n}
& \coker(1-A_E^{mn})^t \arrow{d}{\sigma_{mn} \circ \psi_{mn}} \\
K_0(C^*(E(n))) \arrow{r}{K_0(\tilde j_{n,mn})}
& K_0(C^*(E(mn)))
\end{tikzcd}

\end{center}

\end{lem}

\begin{proof} Let $\eta_n := \sigma_n \circ \psi_n$, and fix $v \in E^0$. Then \begin{align*}
K_0(\tilde j_{n,mn}) (\eta_{n})(\delta_{v} + \im(1 - A_E^n)^t))
&= K_0(\tilde j_{n,mn})([p_{v,n}]_0) \\&= \sum_{\mu \in v E^{<mn}, |\mu| \in n \mathbb{N}} [p_{\mu,mn}]_0
= \sum_{i = 0}^{m-1} \sum_{\mu \in v E^{i n}} [p_{\mu, mn}]_0. \end{align*}

Now, by Remark~\ref{raeburnsktheory}, we have
\begin{align*} (\eta_{mn})\Big(\sum_{i = 0}^{m-1} (A_E^{i n})^t \delta_{v} +  \im(1 - A_E^{mn})^t \Big)
&= \eta_{mn} \Big(\sum_{i = 0}^{m-1} \sum_{\mu \in v E^{i n}} \delta_{s(\mu)} + \im(1- A_E^{mn})^t \Big) \\
&= \sum_{i = 0}^{m-1} \sum_{\mu \in v E^{i n}} (\eta_{mn})(\delta_{s(\mu)} + \im(1- A_E^{mn})^t ) \\
&= \sum_{i = 0}^{m-1} \sum_{\mu \in v E^{i n}} [p_{s(\mu),mn}]_0
= \sum_{i = 0}^{m-1} \sum_{\mu \in v E^{i n}} [p_{\mu,mn}]_0. \qedhere \end{align*} \end{proof}

\begin{cor}
Let $E$ be a row-finite directed graph with no sources and let $n, m \in \NN$.
There exists a homomorphism $\phi_{n,mn}: \coker (1 - A_E^n)^t \to \coker (1 - A_E^{mn})^t$ satisfying $\phi_{n,mn}(\delta_{v} + \im(1-A_E^n)^t) = \sum_{\mu \in v E^{<mn},|\mu| \in n \NN} \delta_{s(\mu),mn} + \im(1-A_E^{mn})^t$ for $v \in E^0$.
\end{cor}

\begin{proof} Let $\eta_n := \sigma_n \circ \psi_n$, and define $\phi_{n,mn}: \coker (1 - A_E^n)^t \to \coker (1 - A_E^{mn})^t$
by $\phi_{n,mn} := \eta_{mn}^{-1} \circ K_0(\tilde j_{n,mn}) \circ \eta_n.$ Let $v \in E^0$. By Remark~\ref{raeburnsktheory}, we have
\begin{align*} \phi_{n,mn}(\delta_{v} + \im(1-A_E^n)^t)
&= (\eta_{mn}^{-1} \circ K_0(\tilde j_{n,mn}) \circ \eta_n)(\delta_{v} + \im(1-A_E^n)^t) \\
&= (\eta_{mn}^{-1} \circ K_0(\tilde j_{n,mn})) ([p_{s(\mu),n}]_0) \\
&= \eta_{mn}^{-1} \Big( \sum_{\mu \in v E^{<mn},|\mu| \in n \NN} [p_{\mu,mn}]_0 \Big) \\
&= \eta_{mn}^{-1} \Big( \sum_{\mu \in v E^{<mn},|\mu| \in n \NN} [p_{s(\mu),mn}]_0 \Big) \\
&= \sum_{\mu \in v E^{<mn},|\mu| \in n \NN} \delta_{s(\mu),mn} + \im(1-A_E^{mn})^t. \qedhere \end{align*} \end{proof}

We now look at direct limits of quotients of abelian groups by their torsion subgroups. We seek to apply the following result to the sequence $(\coker(1-A_E^{n_k})^t, \phi_{n_k,n_{k+1}})_{k=1}^\infty.$

\begin{lem}\label{limtor}
Let $(G_k, \phi_{k,k+1})$ be a directed system of abelian groups. Let $\tor_k := \tor(G_k)$ for each $k \in \NN$, and $\tor_\infty := \tor (\varinjlim G_k)$. For each $k$ there exists a homomorphism $\tilde \phi_{k,k+1}: G_k / \tor_k \to G_{k+1} / \tor_{k+1}$ such that $\tilde \phi_{k,k+1}(g + \tor_k) = \phi_{k,k+1}(g) + \tor_{k+1}$. Moreover, there is an isomorphism \[ \tilde q_\infty: \varinjlim (G_k, \phi_{k,k+1}) / \tor_\infty \to \varinjlim (G_k / \tor_k, \tilde \phi_{k,k+1}) \] such that $\tilde q_\infty(\phi_{k,\infty}(g) + \tor_\infty) = \tilde \phi_{k,\infty}(g + \tor_k)$.
\end{lem}

\begin{proof}

Write $Q_k := G_k / \tor_k$. For each $k \in \NN$, let $q_k: G_k \to Q_k$ be the quotient map. Let $r \in \tor_k$. Then there exists $n \geq 1$ such that $n r = 0$, and then $n \phi_{k,k+1}(r) = \phi_{k,k+1}(n r) =0$. So $\phi_{k,k+1}(\tor_k) \subseteq \tor_{k+1}$, and hence $q_{k+1} \circ \phi_{k,k+1}$ descends to a homomorphism $\tilde \phi_{k,k+1}: Q_k \to Q_{k+1}$ such that $\tilde \phi_{k,k+1}(g + \tor_k) = \phi_{k,k+1}(g) + \tor_{k+1}$ for all $g \in G_k$. So \[ (\tilde \phi_{k+1,\infty} \circ q_{k+1}) \circ \phi_{k,k+1} = \tilde \phi_{k+1,\infty} \circ \tilde \phi_{k,k+1} \circ q_k = \tilde \phi_{k,\infty} \circ q_k\] for all $k \in \NN$. Therefore the universal property of $\varinjlim (G_k,\phi_k)$ gives a homomorphism $q_\infty: \varinjlim (G_k,\phi_{k,k+1}) \to \varinjlim (Q_k,\tilde \phi_{k,k+1})$ satisfing $q_\infty \circ \phi_{k,\infty} = \tilde \phi_{k,\infty} \circ q_k$.

We show that $q_\infty$ descends to a homomorphism satisfying the desired formula. Let $p \in \tor_\infty$. Then there exists $r \in G_k$ and $n \geq 1$ such that $0=np=n\phi_{k,\infty} (r) = \phi_{k,\infty}(nr)$. By \cite[Proposition~6.2.5(ii)]{Rordam2000} we have $\ker \phi_{k,\infty} =\bigcup_{m \ge 0} \ker \phi_{k,k+m}$, so there exists $m \ge 0$ such that $0=\phi_{k,k+m}(nr) = n \phi_{k,k+m}(r)$, giving
$\phi_{k,k+m}(r) \in \tor_{k+m}$. Therefore $q_\infty(p) = q_\infty(\phi_{k,\infty}(r)) = q_\infty(\phi_{k+m,\infty} (\phi_{k,k+m}(r))) = \tilde \phi_{k+m,\infty}(q_{k+m}(\phi_{k,k+m}(r))) =0$. So $q_\infty(\tor_\infty) \subseteq \{0\}$, and hence $q_\infty$ descends to a homomorphism $\tilde q_\infty: \varinjlim (G_k, \phi_{k,k+1}) / \tor_\infty
\to \varinjlim (Q_k, \tilde \phi_{k,k+1})$ satisfying \[ \tilde q_\infty( \phi_{k,\infty}(g) + \tor_\infty) = q_\infty(\phi_{k,\infty}(g)) = \tilde \phi_{k,\infty} (q_k(g)) = \tilde \phi_{k,\infty}(g + \tor_k) \] for all $g \in G_k$.

It remains to show that $\tilde q_\infty$ is an isomorphism. We do this by finding an inverse. As in the first paragraph, we find that $\phi_{k,\infty}(\tor_k) \subseteq \tor_\infty$ since $\phi_{k,\infty}$ is a homomorphism. Therefore $\phi_{k,\infty}$ descends to a homomorphism $\psi_{k,\infty}: Q_k \to \varinjlim (G_k,\phi_{k,k+1}) / \tor_\infty$ satisfying $\psi_{k,\infty} (g + \tor_k) = \phi_{k,\infty}(g) + \tor_\infty$ for each $g \in G_k$. We have \begin{align*} \psi_{k+1,\infty} (\tilde \phi_{k,k+1}(g + \tor_k)) &= \psi_{k+1,\infty}(\phi_{k,k+1}(g)+\tor_{k+1}) = \phi_{k+1,\infty} (\phi_{k,k+1}(g)) +\tor_\infty \\
&= \phi_{k,\infty}(g) + \tor_\infty = \psi_{k,\infty}(g+\tor_k). \end{align*} So $\psi_{k+1,\infty} \circ \tilde \phi_{k,k+1}
= \psi_{k,\infty}$, and hence the universal property of $\varinjlim (Q_k, \tilde \phi_{k})$
gives a homomorphism $\psi: \varinjlim (Q_k , \tilde \phi_{k}) \to \varinjlim(G_k, \phi_{k,k+1} )/ \tor_\infty$ satisfying $\psi(\tilde \phi_{k,\infty} (g + \tor_k)) = \phi_{k,\infty}(g) + \tor_\infty$ for all $g \in G_k$.

We check that $\psi$ is an inverse for $\tilde q_\infty$. Let $g \in G_k$.
Then \begin{align*} \tilde q_\infty( \psi (\tilde \phi_{k,\infty}(g+\tor_k))) &=
\tilde q_\infty(\phi_{k,\infty}(g) + \tor_\infty) = q_\infty(\phi_{k,\infty}(g)) \\
&= \tilde \phi_{k,\infty}(q_k(g)) = \tilde \phi_{k,\infty}(g+\tor_k). \end{align*}
We also have \begin{align*} \psi(\tilde q_\infty (\phi_{k,\infty}(g) + \tor_\infty))
&= \psi(q_\infty(\phi_{k,\infty}(g)) \\
&= \psi(\tilde \phi_{k,\infty}(q_k(g))) = \phi_{k,\infty}(g) + \tor_\infty. \end{align*}
So $\tilde q_\infty \circ \psi$ is the identity on $\tilde \phi_{k,\infty} (Q_k)$ and
$\psi \circ \tilde q_\infty$ is the identity on $\phi_{k,\infty}(G_k) / \tor_\infty$,
and hence by continuity, $\psi$ and $q_\infty$ are mutually inverse.
\end{proof}

For $n \geq 1$, the torsion subgroup of $\coker(1 - A_E^{n})^t$ is
\[ \{ a + \im(1-A_E^{n})^t: a \in \ZZ^{E^0}, ma \in \im(1-A_E^{n})^t
\text{ for some } m \in \NN \}. \] Define \[ T_{n} := \{ a \in \ZZ^{E^0}:
m a \in \im(1-A_E^{n})^t \text{ for some } m \in \NN \}.\] So
$T_n = q_n^{-1}(\tor_n)$ where $q_n: \ZZ^{E^0} \to \coker(1-A_E^n)^t$
is the quotient map.

\begin{prp}\label{torsioninvariant}
Let $E$ be a strongly connected finite directed graph. Suppose $1$ is an eigenvalue of $A_E^t$ and that the only roots of unity that are eigenvalues of $A_E^t$ are the $\Pp_E$-th roots of unity. Let $\omega = (n_k)_{k=1}^\infty$ be a multiplicative sequence and let $ l := \gcd(P_E, \omega)$. Then $T_{n_k} = T_{l}$ for all $k$ such that $\gcd(\Pp_E,n_k) = l$.
\end{prp}

\begin{proof}
Fix $k \ge K$. Let $C := \sum_{i=0}^{n_k/l-1} (A_E^{i l})^t$. We have \begin{align}\label{eqn:determ} (1-A_E^{n_{k}})^t =
(1-A_E^l)^t \Big(\sum_{i=0}^{n_{k}/l-1} (A_E^{il})^t \Big) = (1-A_E^{l})^t C. \end{align}
So $\im(1-A_E^{n_{k}})^t \subseteq \im(1-A_E^{l})^t$. Now take $x \in T_{n_k}$. Then there exists $m \in \NN$ such that
$mx \in \im(1-A_E^{n_k})^t \subseteq \im(1-A_E^{l})^t$. Hence $T_{n_k} \subseteq T_l$.

For the reverse inclusion, take $x \in T_l$. Then $m x = (1 - A_E^{l})^t y$, for some $m \in \NN$ and $y \in \ZZ^{E^0}$.
Equation~\eqref{eqn:determ} gives \begin{align*} (m \det C) x &= (\det C) (1 - A_E^{l})^t y
= (1 - A_E^{l})^t C (\det C) C^{-1} y \\ &= (1 - A_E^{n_{k}})^t (\det C) C^{-1} y \in \im(1 - A_E^{n_{k}})^t. \end{align*}
By Lemma~\ref{cmkdet} $\det C \not = 0$, so $T_l \subseteq T_{n_k}$. \end{proof}

\begin{rmk}
If we could compute $\det C$, we could compute $\det (1-A_E^{n_k})^t$. Then (when $1$ is not an eigenvalue), we could calculate $|K_0(C^*(E,n_k))|$ and try to use Kribs' argument for \cite[Theorem~5.1]{Kribs02} to prove a classification result for the generalised Bunce--Deddens algebras constructed from a finite strongly connected graph whose vertex matrix does not have eigenvalue $1$.
\end{rmk}

\begin{lem}
Let $E$ be a strongly connected finite directed graph. Suppose $1$ is an eigenvalue of $A_E^t$ and that the only roots of unity that are eigenvalues of $A_E^t$ are the $\Pp_E$-th roots of unity. Let $\omega = (n_k)_{k=1}^\infty$ be a multiplicative sequence, and let $l := \gcd(P_E, \omega)$. For each $k$ such that $\gcd(\Pp_E,n_k)=l$, there is an isomorphism $\tau: \coker(1-A_E^{n_k})^t / \tor_{n_k} \to \ZZ^{E^0} / T_l$ satisfying \begin{align}\label{eq:torsion preimage} \tau(a + \im(1-A_E^{n_k})^t + \tor_{n_k}) = a + T_l \end{align} for $a \in \ZZ^{E^0}$.
\end{lem}

\begin{proof}
To see that  the formula \eqref{eq:torsion preimage} is well-defined, suppose $(a + \im(1-A_E^{n_k})^t) + \tor_{n_k} = (b + \im(1-A_E^{n_k})^t) + \tor_{n_k}$, where $a, b \in \ZZ^{E^0}$. Then $a + \im(1-A_E^{n_k})^t = b + \im(1-A_E^{n_k})^t + t$, where $t \in \tor_{n_k}$, that is, $t = c + \im(1-A_E^{n_k})^t$ for some $c \in T_{n_k}$. Then $a - b - c \in \im(1-A_E^{n_k})^t \subseteq \im(1-A_E^{l})^t \subseteq T_l$.  By Proposition~\ref{torsioninvariant}, $c \in T_l$, so $a -b \in T_l$. So there is a map $\tau$ satisfying \eqref{eq:torsion preimage}.

The map $\tau$ is clearly a surjective group homomorphism. To see that it is injective, suppose $a + T_l = b + T_l$ for $a, b \in \ZZ^{E^0}$. We have $a = b + c$, for some $c \in T_l$, and hence $a + \im(1-A_E^{n_k})^t = b + c + \im(1-A_E^{n_k})^t$. So $a + \im(1-A_E^{n_k})^t = b + \im(1-A_E^{n_k})^t + c + \im(1-A_E^{n_k})^t$. By Proposition~\ref{torsioninvariant}, $c \in T_{n_k}$. Therefore $a + \im(1-A_E^{n_k})^t + \tor_{n_k} = b + \im(1-A_E^{n_k})^t + \tor_{n_k}$.
\end{proof}

\begin{cor}\label{cor:torsionfreeiso}
Let $E$ be a strongly connected finite directed graph. Suppose that $1$ is an eigenvalue of $A_E^t$ and that the only roots of unity that are eigenvalues of $A_E^t$ are the $\Pp_E$-th roots of unity. Let $\omega = (n_k)_{k=1}^\infty$ be a multiplicative sequence, let $l := \gcd(P_E, \omega)$. For each $k$ such that $\gcd(\Pp_E,n_k)=l$, there is an isomorphism $\theta_{n_k}: \coker(1-A_E^{n_k})^t / \tor_{n_k} \to \coker(1-A_E^{l})^t / \tor_{l}$ given by $\theta_{n_k} \big((a + \im(1-A_E^{n_k})^t) + \tor_{n_k} \big) = (a + \im(1-A_E^l)^t) + \tor_l$ for $a \in \ZZ^{E^0}$.
\end{cor}

\begin{proof}
Fix $k$ such that $\gcd(\Pp_E,n_k)=l$. The previous Lemma gives an isomorphism $\coker(1-A_E^{n_k})^t / \tor_{n_k} \to \ZZ^{E^0} / T_l$ satisfying
$(a + \im(1-A_E^{n_k})^t) + \tor_{n_k} \mapsto a + T_l,$ where $a \in \ZZ^{E^0}$. The result follows since
$\ZZ^{E^0} / T_l$ is isomorphic to $\coker(1-A_E^l)^t / \tor_l$ via $a + T_l \mapsto a + \im(1-A_E^l)^t + \tor_l$.
We take $\theta_{n_k}$ to be the composition of these isomorphisms.
\end{proof}

We give another description of the torsion-free abelian group $A \big[ \frac{1}{\omega} \big]$ of Lemma~\ref{adjoingroup}.

\begin{lem}\label{limad}
Let $A$ be a free abelian group and let $\omega = (n_k)_{k=1}^\infty$ be a multiplicative sequence, and let $m_k := n_{k+1}/n_k$ for all $k \in \NN$. Define maps $M_{k}: A \to A$ by $M_{k}(a) = m_k \cdot a$, and let $M_{k,\infty}$ be the natural map $A \to \varinjlim(A, M_k)$. There is an isomorphism $\phi: \varinjlim (A, M_{k}) \cong A \big[ \frac{1}{\omega} \big]$ satisfying $\phi(M_{k,\infty}(a)) = a/n_k$, for each $k \in \NN$ and $a \in A$.
\end{lem}

\begin{proof}
Fix $k \in \NN$. Define $j_{k,\infty}: A \to A \big[ \frac{1}{\omega} \big]$ by $j_{k,\infty}(a) = a/n_k$ for $a \in \ZZ$. This $j_{k,\infty}$ is a homomorphism by definition of the operation on $A \big[ \frac{1}{\omega} \big]$. We calculate $j_{k+1, \infty} (M_k(a)) = (m_k \cdot a) / n_{k+1} = (n_{k+1}/n_k) \cdot (a/n_{k+1}) = a / n_k = j_{k,\infty}(a)$. So the universal property of $\varinjlim (A, M_k)$ induces a homomorphism $\phi$ satisfying the desired formula. It remains to check that $\phi$ is an isomorphism. To see that $\phi$ is injective, fix $a \in A$ such that $\phi(M_{k,\infty}(a)) = 0$. Then $a/n_k = 0$, so $a = 0$. To see that $\phi$ is surjective, fix $a/n_k \in A \Big[ \frac{1}{\omega} \Big]$. Then $\phi(M_{k,\infty}(a)) = a/n_k$.
\end{proof}

\begin{prp}\label{mult} Let $E$ be a strongly connected finite directed graph. Suppose that $1$ is an eigenvalue of $A_E^t$ and that the only roots of unity that are eigenvalues of $A_E^t$ are the $\Pp_E$-th roots of unity. Let $\omega = (n_k)_{k=1}^\infty$ be a mulitiplicative sequence, and let $l := \gcd(P_E, \omega)$. Fix $K$ such that $\gcd(P_E, n_K) = l$, and define $\omega' := (n_k')_{k=1}^\infty$ where $n_1' = l$ and $n_k' = n_{K+k-1}$ for $k \ge 2$. For each $k \ge 1$, the map $\phi_{n_k',n_{k+1}'}$ descends to a map $\tilde \phi_{n_k',n_{k+1}'}$ such that the following diagram commutes.

\begin{center} \begin{tikzcd}
\coker(1-A_E^{n_k'})^t / \tor_{n_k'} \arrow{d}{\theta_{n_k'}} \arrow{r}{\tilde \phi_{n_k',n_{k+1}'}}
& \coker(1-A_E^{n_{k+1}'})^t / \tor_{n_{k+1}'} \arrow{d}{\theta_{n_{k+1}'}} \\
\coker(1-A_E^l)^t / \tor_{l} \arrow{r}{M_k'}
& \coker(1-A_E^l)^t / \tor_{l}
\end{tikzcd} \end{center}

\end{prp}

\begin{proof}
Fix $k \ge 1$. Applying the first assertion of Lemma~\ref{limtor} we see that $\phi_{n_k',n_{k+1}'}$ descends to a homomorphism $\tilde \phi_{n_k',n_{k+1}'}:
\coker(1-A_E^{n_k'})^t / \tor_{n_k'} \to \coker(1-A_E^{n_{k+1}'})^t / \tor_{n_{k+1}'},$ satisfying $\tilde \phi_{n_k',n_{k+1}'}(g + \tor_{n_k'}) = \phi_{n_k',n_{k+1}'}(g) + \tor_{n_{k+1}'}$.

Define $B_k := \sum_{i=0}^{m_k' -1} (A_E^{i n_k'} -1)^t$. Note that $B_k + m_k' 1 = \sum_{i = 0}^{m_k'-1} (A_E^{i n_k'})^t$. We have that $(A_E^{in_k'}-1)^t = (A_E^{n_k'}-1)^t (\sum_{j=0}^{i-1} (A_E^{j n_k'})^t)$, so $\im B_k \subseteq \im(1-A_E^{n_k'})^t \subseteq T_{n_k'} = T_l$ by Lemma~\ref{torsioninvariant}. Thus $\im B_k + \im(1-A_E^l)^t \subseteq \tor_l$.

Fix $x \in \ZZ^{E^0}$. By the preceding paragraph, we have \begin{align*}
\theta_{n_{k+1}'} \big(\tilde \phi_{n_k',n_{k+1}'} \big(x + \im(1-A_E^{n_k'})^t + \tor_{n_k'}\big) \big)
&= \theta_{n_{k+1}'} \big(\phi_{n_k',n_{k+1}'}\big(x + \im(1-A_E^{n_k'})^t\big) + \tor_{n_{k+1}'} \big) \\
&= \sum_{i = 0}^{m_k'-1} (A_E^{i n_k'})^t x + \im(1 - A_E^{l})^t + \tor_{l}\\
& = (B_k + m_k' 1)(x) + \im(1 - A_E^l)^t + \tor_l \\
&= (m_k' 1)(x) + \im(1-A_E^l)^t + \tor_l \\
&= M_k'(x + \im(1-A_E^l)^t + \tor_l) \\
&= (M_k' \circ \theta_{n_k'})(x + \im(1-A_E^{n_k'})^t + \tor_{n_k'}). \qedhere \end{align*} \end{proof}

Recall the isomorphism $\rho: \coker(1-A_E^l)^t  \to \bigoplus_{i=1}^l \coker(1-A_E^t)$ of Lemma~\ref{C*(E(l))1} satisfying \[ \rho \big( \delta_v+\im(1-A_E^l)^t \big) = \big(0, \dots, \delta_v+\im(1-A_E^t), \dots, 0 \big),\] where $v \in \Lambda_j$ for some $0 \le j \le l -1$, and $\delta_v + \im(1-A_E^t)$ appears in the $j$-th position.

\begin{lem}\label{C*(E(l))2}
Let $E$ be a strongly connected finite directed graph. Suppose that $1$ is an eigenvalue of $A_E^t$ and that the only roots of unity that are eigenvalues of $A_E^t$ are the $\Pp_E$-th roots of unity. Let $\omega = (n_k)_{k=1}^\infty$ be a multiplicative sequence, and let $l = \gcd(P_E, \omega)$. There is an isomorphism
$\psi: K_0(C^*(E(l))) \to \bigoplus_{i=1}^l K_0(C^*(E))$ such that the following diagram commutes.

\begin{center}
\begin{tikzcd}
\coker(1-A_E^l)^t \arrow{d}{\sigma_l \circ \psi_l} \arrow{r}{\rho}
& \bigoplus_{i=1}^l \arrow{d}{\bigoplus_{i=1}^l \sigma_1} \coker(1-A_E^t) \\
K_0(C^*(E(l)))  \arrow{r}{\psi}
& \bigoplus_{i=1}^l K_0(C^*(E))
\end{tikzcd}
\end{center} Moreover, $\psi \Big(\sum_{\mu \in E^{<l}} [p_{s(\mu),l}]_0 \Big) =
([1_{C^*(E)}]_0, \dots, [1_{C^*(E)}]_0)$.
\end{lem}

\begin{proof}

We define $\psi := (\bigoplus_{i=1}^l \sigma_1) \circ \rho \circ (\sigma_l \circ \psi_l)^{-1}$. Since $\rho, \sigma_l \circ \psi_l$, and $\sigma_1$ are all isomorphisms, so is $\psi$.

We now show that $\psi$ satisfies the second statement. Fix $0 \le i \le l -1$, and $v \in \Lambda_i$. Using Lemma~\ref{identitycokernel} at the second equality, we have
\begin{align*} \rho \Big(\sum_{j=0}^{l-1} (A_E^j)^t \delta_v + \im(1-A_E^l)^t \Big)
&= \Big((A_E^{l-i})^t  \delta_v + \im(1-A_E^t), \dots, (A_E^{l-i-1})^t \delta_v + \im(1-A_E^t) \Big) \\
&= \Big( \delta_v + \im(1-A_E^t), \dots, \delta_v + \im(1-A_E^t) \Big). \end{align*}

Hence,
\begin{align*} \psi \Big(\sum_{\mu \in E^{<l}} [p_{s(\mu),l}]_0 \Big)
&= \Big(\Big(\bigoplus_{i=1}^l \sigma_1\Big) \circ \rho \circ (\sigma_l \circ \psi_l)^{-1}\Big) \Big(\sum_{\mu \in E^{<l}} [p_{s(\mu),l}]_0\Big) \\
&= \Big(\Big(\bigoplus_{i=1}^l \sigma_1\Big) \circ \rho \Big) \Big( \sum_{\mu \in E^{<l}} \delta_{s(\mu)}+\im(1-A_E^l)^t \Big) \Big) \\
&= \Big(\Big(\bigoplus_{i=1}^l \sigma_1\Big) \circ \rho \Big) \Big( \sum_{v \in E^0} \sum_{j=0}^{l-1} (A_E^j)^t \delta_v + \im(1-A_E^l)^t \Big) \\
&= \Big(\bigoplus_{i=1}^l \sigma_1\Big) \Big(\sum_{v \in E^0} \delta_v+\im(1-A_E^t), \dots, \sum_{v \in E^0} \delta_v+\im(1-A_E^t) \Big) \\
&= ([1_{C^*(E)}]_0, \dots, [1_{C^*(E)}]_0). \qedhere
\end{align*} \end{proof}

\begin{proof}[Proof of Theorem~\ref{supchar}]

Fix $K$ such that $\gcd(P_E, n_K) = \gcd(P_E, \omega)$, and let $\omega' = (n_k')_{k=1}^\infty$ where
$n_1' = l$ and $n_k' = n_{K+k-1}$ for $k \ge 2$. Let $m_k'= n_{k+1}' / n_k'$. Since $[\omega] = [\omega']$,
we have a unital isomorphism $C^*(E,\omega) \cong C^*(E, \omega')$ by \cite[Proposition~3.11]{RobertsonRoutSims}. Hence
\[ \big(K_0(C^*(E,\omega)),[1_{C^*(E,\omega)}] \big) \cong \big( K_0(C^*(E, \omega')), [1_{C^*(E,\omega')}] \big).\]
So it suffices to prove the theorem for $\omega'$.

Let $\tor_{\omega'} := \tor \big( \varinjlim \big(K_0(C^*(E(n_k'))), K_0(j_{n_k',n_{k+1}'}) \big) \big)$.
By \cite[Theorem~6.3.2]{Rordam2000} there is an isomorphism \[ K_0(C^*(E, \omega')) \cong \varinjlim
\Big(K_0(C^*(E(n_k'))), K_0(j_{n_k',n_{k+1}'}) \Big) \] satisfying \[ [1_{C^*(E,\omega')}] \mapsto
K_0(j_{n_1', \infty}) \Big(\sum_{\mu \in E^{<n_1'}} [p_{\mu,n_1'}]_0 \Big). \] This isomorphism descends to an
isomorphism  \[ K_0(C^*(E, \omega')) / \tor_{(E,\omega')} \cong \varinjlim \Big(K_0(C^*(E(n_k'))),
K_0(j_{n_k',n_{k+1}'}) \Big) / \tor_{\omega'} \] satisfying \[ [1_{C^*(E,\omega')}]_0 + \tor_{(E,\omega')} \mapsto 
K_0(j_{n_1', \infty}) \Big(\sum_{\mu \in E^{<n_1'}} [p_{\mu,n_1'}]_0 \Big) + \tor_{\omega'}.\]

Let $x := \sum_{\mu \in E^{<n_1'}} \delta_{s(\mu)} \in \ZZ^{E^0}$, and let $\tor_\infty :=
\tor \big( \varinjlim \big(\coker(1- A_E^{n_k'})^t, \phi_{n_k',n_{k+1}'} \big) \big)$. The isomorphisms $(\sigma_{n_k'} \circ \psi_{n_k'})^{-1}$
discussed in Remark~\ref{raeburnsktheory} induce an isomorphism \[ \varinjlim \big(K_0(C^*(E(n_k'))),
K_0(j_{n_k',n_{k+1}'}) \big) / \tor_{\omega'} \cong \varinjlim (\coker(1- A_E^{n_k'})^t, \phi_{n_k',n_{k+1}'}) / \tor_\infty \]
satisfying \begin{align*} K_0(j_{n_1', \infty}) \Big(\sum_{\mu \in E^{<n_1'}} [p_{\mu,n_1'}]_0 \Big) + \tor_{\omega'} 
& \mapsto \phi_{n_1',\infty} \Big( (\sigma_{n_1'} \circ \psi_{n_1'})^{-1}\Big(\sum_{\mu \in E^{<n_1'}} [p_{\mu,n_1'}]_0 \Big) \Big) + \tor_\infty \\
&= \phi_{n_1',\infty}(x + \im(1-A_E^{n_1'})^t) +\tor_\infty. \end{align*}

By Lemma~\ref{limtor} there is an isomorphism \[ \varinjlim (\coker(1- A_E^{n_k'})^t, \phi_{n_k',n_{k+1}'}) / \tor_\infty \cong \varinjlim
(\coker(1- A_E^{n_k'})^t / \tor_{n_k'}, \tilde \phi_{n_k',n_{k+1}'}) \] satisfying $\phi_{n_1',\infty}(x + \im(1-A_E^{n_1'})^t) + \tor_\infty
\mapsto \tilde \phi_{n_1',\infty}( x + \im(1-A_E^{n_1'})^t + \tor_{n_1'})$.

By Proposition~\ref{mult} there is an isomorphism \[ \varinjlim (\coker(1- A_E^{n_k'})^t / \tor_{n_k'}, \tilde \phi_{n_k',n_{k+1}'})
\cong \varinjlim (\coker(1- A_E^l)^t / \tor_{l}, M_{n_k'} ) \] satisfying $\tilde \phi_{n_1', \infty}(x + \im(1-A_E^{n_1'})^t +
\tor_{n_1'}) \mapsto M_{n_1',\infty}(x + \im(1-A_E^{l})^t+\tor_{l}).$

By Lemma~\ref{limad} there is an isomorphism \[ \varinjlim (\coker(1- A_E^l)^t / \tor_{l}, M_{n_k'} )
\cong \big(\coker(1- A_E^l)^t / \tor_{l} \big) \Big[\frac{1}{\omega'} \Big] \] satisfying $m_{n_1',\infty}
(x + \im(1-A_E^{l})^t +\tor_{l}) \mapsto (x + \im(1-A_E^{l})^t +\tor_{l}) /n_1'$.

The isomorphism $\eta_l := \sigma_l \circ \psi_l: \coker(1- A_E^l)^t \to K_0(C^*(E(l)))$ of Remark~\ref{raeburnsktheory} descends
to an isomorphism $\tilde \eta_l: \coker(1- A_E^l)^t /\tor_l \to K_0(C^*(E(l))) / \tor_{E(l)}$.
This $\tilde \eta_l$ induces an isomorphism \[ \big(\coker(1- A_E^l)^t / \tor_{l} \big)
\Big[\frac{1}{\omega'} \Big] \cong \big( K_0(C^*(E(l))) / \tor_{E(l)} \big) \Big[\frac{1}{\omega'} \Big], \]
satisfying \begin{align*} \big(x + \im(1-A_E^{l})^t +\tor_{l}\big) /n_1' & \mapsto \tilde \eta_l
(x + \im(1-A_E^{l})^t +\tor_{l}) /n_1' \\ &= \Big(\sum_{\mu \in E^{<l}} [p_{s(\mu), l}]_0 +\tor_{E(l)} \Big) /n_1'. \end{align*}

The isomorphism of Lemma~\ref{C*(E(l))2} descends to an isomorphism
$K_0(C^*(E(l))) / \tor_{E(l)} \to \bigoplus_{i=1}^l K_0(C^*(E)) / \tor_E,$ and this induces
an isomorphism \[ \big( K_0(C^*(E(l))) / \tor_{E(l)} \big)
\Big[\frac{1}{\omega'} \Big] \cong \Big(\bigoplus_{i=1}^l K_0(C^*(E))  / \tor_E \Big) \Big[\frac{1}{\omega'} \Big], \]
satisfying \begin{align*} \Big(\sum_{\mu \in E^{<l}} [p_{s(\mu), l}]_0 +\tor_{E(l)} \Big) /n_1'
&\mapsto \tilde \psi \Big(\sum_{\mu \in E^{<l}} [p_{s(\mu), l}]_0 +\tor_{E(l)} \Big) /n_1' \\
& = ([1_{C^*(E)}]_0 + \tor_E, \dots, [1_{C^*(E)}]_0+ \tor_E) / l, \end{align*} since $n_1'=l$.

Composing the isomorphisms of the previous seven paragraphs gives an isomorphism \[ \Psi:
K_0(C^*(E,\omega'))/ \tor_{(E,\omega')} \to \bigoplus_{i=1}^l \big(K_0(C^*(E)) / \tor_{E} \big) \Big[ \frac{1} {\omega' } \Big] \]
satisfying $\Psi([1_{C^*(E,\omega')}]) = ([1_{C^*(E)}]_0 + \tor_E, \dots, [1_{C^*(E)}]_0+\tor_E) / l$. \end{proof}

\begin{rmk}
In the proof of Theorem~\ref{supchar}, we needed to apply Lemma~\ref{C*(E(l))2} to relate the torsion-free component of $K_0(C^*(E(l)))$ back to the torsion-free component of $K_0(C^*(E))$. This uses Corollary~\ref{C*(E(l))1}, which requires Lemma~\ref{isovarup}, where it is crucial that the power of $A_E^t$ in the term $(1-A_E^l)^t$ matches the number of equivalence classes for the equivalence relation $\sim_l$. We also needed to apply Corollary~\ref{cor:torsionfreeiso} to obtain an isomorphism between the torsion-free component of $K_0(C^*(E(l)))$ and the torsion-free component of $K_0(C^*(E(n_k)))$ for all $k$ such that $\gcd(\Pp_E, n_k) = l$. This uses Lemma~\ref{torsioninvariant} which depends on Lemma~\ref{cmkdet} explaining why we require that the only roots of unity that are eigenvalues of $A_E^t$ are the $\Pp_E$-th roots of unity.
\end{rmk}

\section{Classification of $C^*(E,\omega)$}\label{sec:class}

In this section we use Theorem~\ref{supchar} to prove the following isomorphism theorem.

 \begin{thm}\label{thm:classification}
Fix a strongly connected finite directed graph $E$.
Let $\omega = (n_k)_{k=1}^\infty$ and $\omega' = (n_k')_{k=1}^\infty$ be multiplicative sequences.
Suppose $1$ is an eigenvalue of $A_E^t$ and that the only roots of unity that are eigenvalues of $A_E^t$ are the
$\Pp_E$-th roots of unity. Then $C^*(E,\omega)
\cong C^*(E,\omega')$ if and only if $[\omega] = [\omega']$.
\end{thm}

To prove this theorem we need some preliminary results.

\begin{lem}\label{lcmomega}
Let $D \subseteq \NN$. Suppose $|D| = m$ for some $1 \le m \le \infty$, enumerate $D$ in increasing order, $(d_1, d_2, \dots, d_m)$, and define a nondcreasing sequence $\lcm(D)$ by \[ \lcm(D) := (d_1, \lcm(d_1,d_2), \lcm(d_1, d_2, d_3), \dots\lcm(d_1,d_2, \dots, d_m), \lcm(d_1,d_2, \dots, d_m), \dots).\] Then $\lcm(D)$ is a multiplicative sequence such that $d_k | \lcm(D)$ for all $1 \le k \le m$. Moreover, if $\omega = (n_k)_{k=1}^\infty$ is another multiplicative sequence such that $d_k | \omega$ for all $1 \le k \le m$, then $[\lcm(D)]$ divides $[\omega]$.
\end{lem}

\begin{proof}

Clearly $\lcm(D)_k | \lcm(D)_{k+1}$ for each $k \ge 1$. It is also clear that, for each $1 \le k \le m$, $d_k | \lcm(D)_l$ for all $l \geq k$, and so $d_k | \lcm(D)$.

For the final statement, fix $\omega$ such that $d_k | \omega$ for each $1 \le k \le m$. For each $1 \le k \le m$, there exist natural numbers $l_1, \dots, l_k$ such that $d_1 | n_{l_1}, \dots, d_k | n_{l_k}$. Let $l(k) = \max \{ l_1, \dots, l_k \}$. Then $d_i|n_{l(k)}$ for each $1 \le i \le k$, so $\lcm(d_1, \dots, d_k)|n_{l(k)}.$ \end{proof}

If $A$ is a free abelian group, $a \in A$ and $n \geq 1$, we write $n | a$ if there exists $a' \in A$ such
that $n a' = a$.

\begin{thm}\label{omegarecovery}

Fix a strongly connected finite directed graph $E$, and a generalised Bunce--Deddens
algebra $C^*(E,\omega)$. Suppose that the only roots of unity that are eigenvalues of $A_E^t$ are the $\Pp_E$-th roots of unity.
Set \[ D := \{ n \ge 1: n | \big( [1_{C^*(E,\omega)}]_0 + \tor_{(E,\omega)} \big) \in K_0(C^*(E,\omega)) / \tor_{(E,\omega)} \}\] and let \[d := \lcm \{n \ge 1 :n | \big( [1_{C^*(E)}]_0 + \tor_E \big) \in K_0(C^*(E)) / \tor_E \}.\] Then $[\omega] = [l \cdot\lcm(D)] / d.$

\end{thm}

\begin{proof}

There is an isomorphism $\theta: K_0(C^*(E)) / \tor_E \to \ZZ^N$, where $N = \rank K_0(C^*(E))$.
Let $(u_1, \dots, u_N) := \theta([1_{C^*(E)}]_0 + \tor_E) \in \ZZ^N$.

We claim that $\gcd(u_1, \dots, u_N) = d$. Let $e_1, \dots, e_N$ be the generators of $\ZZ^N$,
and let $n \ge 1$ such that $n | u_i$ for each $1 \le i \le N$. Then $n$ divides $\sum_{i=1}^N u_i
\theta^{-1}(e_i) = \theta^{-1}(u_1, \dots, u_N) = [1_{C^*(E)}]_0 + \tor_E$. So $n | d$,
and hence $\gcd(u_1,\dots,u_N) | d$.

Now, fix $n \ge 1$ such that $n | ([1_{C^*(E)}]_0 + \tor_E)$. Then there exists $a \in K_0(C^*(E))$
such that $n a + \tor_E = [1_{C^*(E)}]_0 + \tor_E$. We then have that $n \theta(a + \tor_E)
= (u_1, \dots, u_N)$. So $n$ is a common divisor of $u_1, \dots, u_N$, and hence
$n|\gcd(u_1, \dots, u_N)$. So $\gcd(u_1, \dots, u_N)$ is a common multiple of
$\{n \ge 1:n | \big( [1_{C^*(E)}]_0 + \tor_E \big) \in K_0(C^*(E)) / \tor_E \}$,
giving $d | \gcd(u_1, \dots, u_N)$, and so $\gcd(u_1, \dots, u_N) = d$.

Next we claim that for $n \ge 1$, we have $n | \lcm(D)$ if and only if $n \in D$.
If $n \in D$, it is clear that $n|\lcm(D)$. For the other direction, suppose $n | \lcm(D)$.
Then there is an $i \ge 1$ such that $n | \lcm(d_1, \dots, d_i)$. Since $d_1, \dots, d_i \in D$,
we have that $\lcm(d_1, \dots, d_i)$ divides $[1_{C^*(E,\omega)}]_0 + \tor_{(E,\omega)}$,
and so $n \in D$.

We now show that $[\lcm(D)]$ divides $[d\omega/l]$. Fix $n \ge 1$.
Then \begin{align*} n \in D
& \iff n | \big( [1_{C^*(E, \omega)}]_0
+ \tor_{(E,\omega)} \big) \in K_0(C^*(E,\omega)) / \tor_{(E,\omega)} \\
& \iff n | \big([1_{C^*(E))}]_0 + \tor_E, \dots, [1_{C^*(E)}]_0+ \tor_E \big) / l \in
\bigoplus_{i = 1}^l \Big( K_0(C^*(E)) / \tor_E \Big) \Big[\frac{1}{\omega} \Big] \\
& \iff n | \big( [1_{C^*(E)}]_0 + \tor_E \big)/l \in \Big( K_0(C^*(E)) / \tor_E \Big) \Big[\frac{1}{\omega} \Big] \\
& \iff n | (u_1, \dots, u_N)/l \in \bigoplus_{i=1}^N \ZZ \Big[\frac{1}{\omega} \Big] \\
& \iff n | (d/l) \in \ZZ \Big[\frac{1}{\omega} \Big] \\
& \iff n | 1 \in \ZZ \Big[ \frac{1}{(d \omega)/l} \Big] \\
& \iff n |(d \omega/l). \end{align*} Hence $n | d \omega$ for all $n \in D$,
and so $[\lcm(D)]$ divides $[d \omega/l]$ by Lemma~\ref{lcmomega}.

To see that $[d \omega/l]$ divides $[\lcm(D)]$, fix $k \ge 1$. We have that $n_k | 1 \in \ZZ \big[ \frac{1}{\omega}
\big]$, so $(d n_k/l) | (d/l) \in \ZZ \big[ \frac{1}{\omega} \big]$. The above string of implications gives
us $(d n_k/l) | \lcm(D)$ for each $k \ge 1$, so $[d \omega/l]$ divides $[\lcm(D)]$, and the result follows.
\end{proof}

We now prove Theorem~\ref{thm:classification}.

\begin{proof}[Proof of Theorem~\ref{thm:classification}]

Suppose that $[\omega] = [\omega']$. Then $C^*(E,\omega) \cong C^*(E,\omega')$ by \cite[Proposition~3.11]{RobertsonRoutSims}.

Now suppose that $C^*(E,\omega) \cong C^*(E,\omega')$. Let $l = \gcd(\Pp_E,\omega)$ and $l' = \gcd(\Pp_E,\omega)$. Since $C^*(E,\omega) \cong C^*(E,\omega')$, the number of summands in Theorem~\ref{supchar} must be equal, so $l = l'$.

Let $d$ be as in Theorem~\ref{omegarecovery}. Let \[ D := \{ n \ge 1: n |
\big( [1_{C^*(E,\omega)}]_0 + \tor_{(E,\omega)} \big) \in K_0(C^*(E,\omega)) / \tor_{(E,\omega)} \}\]
and let \[ D' := \{ n \ge 1: n | \big( [1_{C^*(E,\omega')}]_0 + \tor_{(E,\omega')} \big) \in
K_0(C^*(E,\omega')) / \tor_{(E,\omega')} \}.\]

Fix $n \ge 1$. Since $C^*(E,\omega) \cong C^*(E,\omega')$, we have that $n$ divides
$[1_{C^*(E,\omega)}]_0 + \tor_{(E,\omega)}$ precisely when $n$ divides $[1_{C^*(E,\omega')}]_0
+ \tor_{(E,\omega')}$, so $D = D'$. By Theorem~\ref{omegarecovery} we have that $[\omega]
= [l \cdot \lcm (D) ]/d = [l \cdot \lcm(D')]/d = [\omega']$.
\end{proof}
\begin{rmk}
Theorem~\ref{thm:classification} says that for a given graph $E$ and $[\omega] \not = [\omega']$, we
have $C^*(E,\omega) \not = C^*(E,\omega')$. One might ask whether this can be extended to say that
given graphs $E$ and $F$ and given $[\omega] \not = [\omega']$, we must have $C^*(E,\omega) \not =
C^*(F,\omega')$. The following example demonstrates that the answer is no.
Let $C_1$ be the graph consisting of a single vertex connected by a single loop and let $C_3$ be the graph
with three vertices connected by a single cycle. Let $\omega = (3,6, 12, 24, \dots)$
and let $\omega' = (1,2,4, 8, 16, \dots)$. Note that $\omega = 3 \, \omega'$. Since $C_1(3) =C_3$,
we have that $C^*(C_1,\omega) \cong C^*(C_3,\omega')$. This illustrates why Theorem~\ref{thm:classification}
applies only to generalised Bunce--Deddens algebras constructed from the same graph.
\end{rmk}

\section{Acknowledgments}

The results in this article are from my PhD thesis. Thanks to my PhD supervisors Aidan Sims and Dave Robertson for their guidance and support during my PhD and during the writing of this article. It has been great learning from such generous and talented mathematicians. Thanks to Gunar Restorff for pointing out an error in Lemma~\ref{dirsuml}. Thanks to Mike Boyle for bringing \cite{KimOrmesRoush} to my attention and for a helpful email conversation about the spectra of nonnegative integer matrices. Thanks to Toke Meier Carlsen for helpful conversations.

\end{document}